\documentclass[10pt, a4paper, reqno, oneside]{amsart}

\usepackage{amsfonts, latexsym, rawfonts, amsmath, amssymb, amsthm, mathrsfs, lscape, mathtools, enumerate, enumitem, scalefnt, multicol, relsize, mathbbol}
\usepackage[all]{xy}
\usepackage{graphicx}
\usepackage[english]{babel}
\usepackage[latin1]{inputenc}
\usepackage{array, tabularx}
\usepackage{textcomp}
\setlist[itemize]{noitemsep, topsep=1pt, leftmargin=20pt}
\usepackage{xfrac}
\usepackage{todonotes}

\usepackage{fullpage}

\usepackage{mathbbol}
\DeclareSymbolFontAlphabet{\amsmathbb}{AMSb}

\newcommand\bcdot{\ensuremath{
  \mathchoice
   {\mskip\thinmuskip\lower0.2ex\hbox{\scalebox{1.6}{$\cdot$}}\mskip\thinmuskip}}
   {\mskip\thinmuskip\lower0.2ex\hbox{\scalebox{1.6}{$\cdot$}}\mskip\thinmuskip}
   {\lower0.3ex\hbox{\scalebox{1.2}{$\cdot$}}}
   {\lower0.3ex\hbox{\scalebox{1.2}{$\cdot$}}}
}

\usepackage[hypertexnames=false,
backref=page,
    pdftex,
    pdfpagemode=UseNone,
    breaklinks=true,
    extension=pdf,
    colorlinks=true,
    linkcolor=blue,
    citecolor=blue,
    urlcolor=blue,
]{hyperref}

\usepackage[top=1.5in, bottom=1in, left=0.9in, right=0.9in]{geometry}

\theoremstyle{plain}
\newtheorem{theo}{Theorem}[section]
\newtheorem{lemma}[theo]{Lemma}
\newtheorem{theorem}[theo]{Theorem}
\newtheorem{corollary}[theo]{Corollary}
\newtheorem{proposition}[theo]{Proposition}

\theoremstyle{definition}

\newtheorem{remark}[theo]{Remark}

\theoremstyle{plain}
\newtheorem{thmint}{Theorem}

\newtheorem{corint}[thmint]{Corollary}

\numberwithin{equation}{section}

\allowdisplaybreaks

\title{A moment map for twisted-Hamiltonian vector fields\\
on locally conformally K\"ahler manifolds}

\author{Daniele Angella}
\address[Daniele Angella]{Dipartimento di Matematica e Informatica ``Ulisse Dini''\\
Universit\`a degli Studi di Firenze,
viale Morgagni 67/a,
50134 Firenze, Italy}
\email{daniele.angella@unifi.it}
\email{daniele.angella@gmail.com}

\author{Simone Calamai}
\address[Simone Calamai]{Dipartimento di Matematica e Informatica ``Ulisse Dini''\\
Universit\`a degli Studi di Firenze,
viale Morgagni 67/a,
50134 Firenze, Italy}
\email{simone.calamai@unifi.it}
\email{simocala@gmail.com}

\author{Francesco Pediconi}
\address[Francesco Pediconi]{Department of Mathematics, Aarhus University, Ny Munkegade 118, 8000 Aarhus C, Denmark}
\email{francesco.pediconi@math.au.dk}
\email{francesco.pediconi@gmail.com}

\author{Cristiano Spotti}
\address[Cristiano Spotti]{Department of Mathematics, Aarhus University, Ny Munkegade 118, 8000 Aarhus C, Denmark}
\email{c.spotti@math.au.dk}

\makeatletter
\@namedef{subjclassname@2020}{\textup{2020} Mathematics Subject Classification}
\makeatother

\thanks{
The first-named, second-named, third-named authors are supported by project PRIN2017 ``Real and Complex Manifolds: Topology, Geometry and holomorphic dynamics'' (code 2017JZ2SW5), and by GNSAGA of INdAM.
The third-named and fourth-named authors are supported by Villum Fonden Grant 0019098}

\begin{document}

\begin{abstract}
We extend the classical Donaldson-Fujiki interpretation of the scalar curvature as moment map in K\"ahler Geometry to the wider framework of locally conformally K\"ahler Geometry.
\end{abstract}

\maketitle

\section*{Introduction}

By the foundational work of Fujiki \cite{fujiki} and Donaldson \cite{donaldson}, it is known that the scalar curvature of K\"ahler metrics arise as a moment map for an infinite-dimensional Hamiltonian action. More precisely, given a compact symplectic manifold $(M,\omega_o)$, the space $\mathcal{J}_{\rm alm}(\omega_o)$ of compatible almost complex structures can be endowed with a natural structure of infinite-dimensional K\"ahler manifold $(\mathbb{J},\mathbb{\Omega})$ that is invariant under the action of the automorphism group ${\rm Aut}(M,\omega_o)$ by pullback. Here, an almost complex structure $J$ is said to be compatible if $\omega_o(\_,J\_)$ is a $J$-almost Hermitian (and hence almost K\"ahler) metric. This action preserves the analytic subset $\mathcal{J}(\omega_o) \subset \mathcal{J}_{\rm alm}(\omega_o)$ of integrable almost complex structures and the map ${\rm scal}: \mathcal{J}_{\rm alm}(\omega_o) \to \mathcal{C}^{\infty}(M,\mathbb{R})$, that assigns to any $J \in \mathcal{J}_{\rm alm}(\omega_o)$ the scalar curvature of the metric $\omega_o(\_,J\_)$, is smooth, ${\rm Aut}(M,\omega_o)$-equivariant and it verifies
$$
\int_M {\rm d}\,{\rm scal}|_J(v)\,h_X\,\omega_o^n = -\tfrac12\,\mathbb{\Omega}_J(X^*_J,v) \quad \text{ for any } J \in \mathcal{J}(\omega_o) \, ,\,\, v \in T_J\mathcal{J}_{\rm alm}(\omega_o) \, , \,\, X \in \mathfrak{ham}(M,\omega_o) \,\, .
$$
Here, $\mathfrak{ham}(M,\omega_o)$ denotes the Lie subalgebra of Hamiltonian vector fields, {\it i.e.\ }those infinitesimal automorphisms $X \in \mathfrak{aut}(M,\omega_o)$ such that $X \lrcorner \omega_o = {\rm d}h_X$ for some $h_X \in \mathcal{C}^{\infty}(M,\mathbb{R})$, uniquely determined up to a constant, and $X^*$ denotes the fundamental vector field associated to $X$. For this reason, we say that ${\rm scal}$ is a {\it moment map for the action of ${\rm Ham}(M,\omega_o)$ on $\mathcal{J}(\omega_o)$}, where ${\rm Ham}(M,\omega_o)$ denotes the Hamiltonian diffeomorphism group.

This fact has deep consequences in K\"ahler Geometry. In particular, the constant scalar curvature K\"ahler metrics arise as zeroes of a moment map equation. More precisely, the problem of finding constant scalar curvature K\"ahler metrics in a fixed K\"ahler class is reduced to finding a zero for the moment map in the orbit of the complex structure under the complexified action. Inspired by the Hilbert-Mumford criterion in Geometric Invariant Theory in finite-dimension, this leads to the notion of K-stability, see {\it e.g.\ }\cite{szekelyhidi}.

\smallskip

This paper is an attempt to generalize this moment map framework to the non-K\"ahler setting.
More precisely, we deal with a compact almost symplectic manifold $(M,\omega)$ with ${\rm d}\omega \neq 0$, admitting compatible complex structures.
We notice here that, in the literature, there appeared other works concerning this problem, see {\itshape e.g.\ }\cite{Apos-Maschl, FutHatOrn}, which are distinct from our approach. For other appearance of moment maps for group actions on locally conformally symplectic or Vaisman manifolds, see \cite{haller-rybicki, gini-ornea-parton, stanciu}. However, we stress that the group action that we consider lives in an infinite-dimensional K\"ahler setting. Recently, Garc\'ia-Prada and Salamon \cite{garciaprada-salamon} described a setting where the moment map is given by the Ricci form and studied moment map interpretations of the K\"ahler-Einstein condition. As the anonymous referee suggested, it would be interesting to try to extend this construction to our locally conformally K\"ahler setting.

One of the first natural non-K\"ahler cases of study \cite{gray-hervella} is given by the so-called {\it locally conformally symplectic} condition, {\it i.e.\ }we ask that ${\rm d}\omega = \theta \wedge \omega$ for a given closed $1$-form $\theta$. Notice that, by introducing the twisted exterior differential operator ${\rm d}_{\theta} \coloneqq {\rm d} - \theta \wedge$, the previous condition can be written as ${\rm d}_{\theta}\omega = 0$. This nomenclature is due to the following characterization: $\omega$ is locally conformally symplectic if and only if, for any point $x \in M$, $\omega$ is locally conformal to a symplectic form in a neighborhood of $x$. This shows clearly that the locally conformally symplectic condition is actually a property for the whole conformal class $[\omega]$ of $\omega$.
Moreover, there is a further relation with the symplectic framework: indeed, $(M,\omega)$ admits a unique minimal covering map $\pi: \widehat{M} \to M$ such that $\pi^*\omega$ is globally conformal to a symplectic form $\widehat{\omega}_o$ on $\widehat{M}$ and the deck transformation group acts by homotheties on $(\widehat{M},\widehat{\omega}_o)$ (see {\it e.g.\ }\cite[Section 2]{gini-ornea-parton-piccinni} and \cite[Section 2.1]{bgp}). Moreover, a vector field $X \in \mathfrak{X}(M)$ lifts through $\pi$ to an infinitesimal automorphism $\pi^*X \in \mathfrak{aut}(\widehat{M},\widehat{\omega}_o)$ if and only if ${\rm d}_{\theta}(X \lrcorner \omega)=0$ \cite[Proposition 3.3]{bgp} and $\pi^*X \in \mathfrak{ham}(\widehat{M},\widehat{\omega}_o)$ if and only if $X \lrcorner \omega$ is ${\rm d}_{\theta}$-exact \cite[Corollary 3.10]{bgp}. Therefore, it is natural to look at the action of the generalized Lie transformation group ${\rm Aut}^{\star}(M,[\omega])$ generated by the Lie algebra of {\it special conformal vector fields of $(M,\omega)$}
$$
\mathfrak{aut}^{\star}(M,[\omega]) := \big\{ X \in \mathfrak{X}(M) : {\rm d}_{\theta}(X \lrcorner \omega) = 0 \big\},
$$
(see Section \ref{sect:defG1} and Section \ref{sect:defG2}) on the space $\mathcal{J}_{\rm alm}(\omega)$ of compatible almost complex structures on $(M,\omega)$. As in the K\"ahler setting, one may expect the existence of a moment map, related to the scalar curvature and the $1$-form $\theta$, when one restricts $\mathfrak{aut}^{\star}(M,[\omega])$ to the Lie subalgebra of {\it twisted-Hamiltonian vector fields}
$$
\mathfrak{ham}(M,[\omega]) \coloneqq \big\{ X \in \mathfrak{X}(M) : X \lrcorner \omega = {\rm d}_{\theta}h \text{ for some } h \in \mathcal{C}^{\infty}(M,\mathbb{R}) \big\},
$$
and $\mathcal{J}_{\rm alm}(\omega)$ to the analytic subset $\mathcal{J}(\omega)$ of integrable almost complex structures. However, two main difficulties occur in this picture. The first one is that the group ${\rm Aut}^{\star}(M,[\omega])$ preserves just the conformal class $[\omega]$, and so the map ${\rm scal}: \mathcal{J}_{\rm alm}(\omega) \to \mathcal{C}^{\infty}(M,\mathbb{R})$ that associates to any $J \in \mathcal{J}_{\rm alm}(\omega)$ the Riemannian scalar curvature of $\omega(\_,J\_)$ turns out to be non-${\rm Aut}^{\star}(M,[\omega])$-equivariant. The second one is that, during the linearization procedure of the map scal, the term $D(\theta^{\sharp})$ appears, where $D$ denotes the Levi-Civita connection and $\_^{\sharp}$ the metric duality, and it seems to us that it cannot be handled via an absorption scheme by adding extra terms depending on $\theta$ itself.

In order to overcome these issues, we impose a further symmetry. More precisely, we consider the {\it symplectic dual of $\theta$}, {\it i.e.\ }the unique vector field $V \in \mathfrak{X}(M)$ such that $V \lrcorner \omega = \theta$, and we assume that there exists a compact, connected Lie group $\mathsf{K}$ acting effectively on $M$ such that its action preserves $\omega$, is of {\it Lee type}, {\it i.e.\ }the flow of $V$ is a 1-parameter subgroup contained in the center of $\mathsf{K}$, and is {\it twisted-Hamiltonian}, {\it i.e.\ }the Lie algebra $\mathfrak{k} \coloneqq {\rm Lie}(\mathsf{K})$ is contained in $\mathfrak{ham}(M,[\omega])$ (see Section \ref{sect:actionK}). Then, it turns out that the group ${\rm Aut}^{\star}(M,[\omega])^{\mathsf{K}} \coloneqq {\rm Aut}^{\star}(M,[\omega]) \cap {\rm Diff}(M)^{\mathsf{K}}$ preserves the $2$-form $\omega$ (see \eqref{eq:autTpresomega}) and that, for any $J \in \mathcal{J}(\omega)^{\mathsf{K}} \coloneqq \mathcal{J}(\omega) \cap \mathcal{J}_{\rm alm}(\omega)^{\mathsf{K}}$, the endomorphism $D(\theta^{\sharp})$ is orthogonal to the tangent space $T_J\mathcal{J}_{\rm alm}(\omega)^{\mathsf{K}}$ (see \eqref{eq:lin-scal-dim(4)}, \eqref{eq:lin-scal-dim(5)} and \eqref{eq:lin-deltaomega-dim(4)}). Here, we denoted by ${\rm Diff}(M)^{\mathsf{K}}$ the subgroup of diffeomorphisms that commute with $\mathsf{K}$ and by $\mathcal{J}_{\rm alm}(\omega)^{\mathsf{K}}$ the subspaces of $\mathsf{K}$-invariant, compatible almost complex structures on $(M,\omega)$. Therefore, if we denote by ${\rm Ham}(M,[\omega])$ the generalized Lie transformation group generated by $\mathfrak{ham}(M,[\omega])$ (see Section \ref{sect:defG1} and Section \ref{sect:defG2}), set ${\rm Ham}(M,[\omega])^{\mathsf{K}} \coloneqq {\rm Ham}(M,[\omega]) \cap {\rm Diff}(M)^{\mathsf{K}}$, and let $\mathcal{C}^{\infty}(M;\mathbb{R})^{\mathsf{K}}$ be the space of $\mathsf{K}$-invariant smooth functions, the previous arguments allow us to prove the main result of this paper, that is

\begin{thmint}[see Proposition \ref{prop:importante}, Proposition \ref{prop:twHamaction} and Theorem \ref{thm:moment-map}] \label{thm:a}
Let $(M^{2n},\omega)$ be a compact, connected, $2n$-dimensional, smooth manifold endowed with a locally conformally symplectic structure with non-exact Lee form $\theta$.
Assume that there exists a compact, connected Lie group $\mathsf{K}$ acting effectively on $M$ such that its action preserves $\omega$ \ref{(k1)}, is of Lee type \ref{(k2)} and is twisted-Hamiltonian \ref{(k3)}. Then, the map
$$
\mu \colon \mathcal{J}(\omega)^{\mathsf{K}} \to \mathcal{C}^{\infty}(M;\mathbb{R})^{\mathsf{K}} \,\, , \quad \mu \coloneqq {\rm scal}^{\rm Ch} +n\,{\rm d}^*\theta
$$
is a {\em moment map} for the action of ${\rm Ham}(M,[\omega])^{\mathsf{K}}$ on $\mathcal{J}(\omega)^{\mathsf{K}}$.
\end{thmint}

Here, we denoted by ${\rm scal}^{\rm Ch}$ the map that assigns to any $J \in \mathcal{J}(\omega)^{\mathsf{K}}$ the {\it Chern-scalar curvature} of the locally conformally K\"ahler metric $\omega(\_,J\_)$ (see Appendix \ref{subsec:appCh}). As expected, the moment map framework developed in Theorem \ref{thm:a} leads to the existence of a {\it Futaki-type invariant}. More precisely, if we denote by $\kappa: \mathfrak{ham}(M,[\omega])^{\mathsf{K}} \to \mathcal{C}^{\infty}(M,\mathbb{R})^{\mathsf{K}}$ the linear isomorphism that verifies $X \lrcorner \omega = {\rm d}_\theta (\kappa(X))$ for any $X \in \mathfrak{ham}(M,[\omega])^{\mathsf{K}}$ (see \eqref{eq:kappa} and Proposition \ref{prop:importante}), and by $\mathfrak{z}(\mathfrak{k})$ the center of the Lie algebra $\mathfrak k$, then the following corollary holds true.

\begin{corint}[see Corollary \ref{cor:futaki}] \label{cor:b}
Under the same hypotheses of Theorem \ref{thm:a}, the value
$$
\underline{\mu} \coloneqq \frac{\int_M \mu(J) \, \omega^n}{\int_M \omega^n}
$$
and the map
$$
\mathcal{F} \colon \mathfrak{z}(\mathfrak{k}) \to \mathbb{R} \,\, , \quad
\mathcal{F}(X) \coloneqq \int_M (\mu(J)-\underline{\mu}) \kappa(X) \, \omega^n
$$
are independent of $J$, in the connected components of ${\mathcal J}(\omega)^{\mathsf{K}}$.
\end{corint}

We expect that our results could possibly lead to new notions of {\it canonical metrics} and {\it stability conditions} in the context of {\it locally conformally K\"ahler Geometry}, playing an analogue role to the constant scalar curvature K\"ahler metrics and K-stability in K\"ahler Geometry. Note that the equation $\mu(J) = \underline{\mu}$ arising from Theorem \ref{thm:a} and \ref{cor:b} does not reduce to the constant scalar curvature K\"ahler equation on the minimal symplectic covering $\widehat M$ (see Remark \ref{rem:csck-covering}). However, in the special case when $J$ is {\it Vaisman}, {\it i.e.\ }$D(\theta^{\sharp})=0$ at $J$, then ${\rm d}^*\theta=0$ at $J$ and so we recover the constant (Chern) scalar curvature equation on $M$. 
\medskip

The paper is organized as follows.
In Section \ref{sect:prel}, we summarize some basic facts on locally conformally symplectic and locally conformally K\"ahler geometry.
In Section \ref{sect:twH-torus}, we recall some notions about generalized Lie transformation groups and torus actions.
In Section \ref{sect:Jomega}, we study the properties of the space $\mathcal{J}_{\rm alm}(\omega)^{\mathsf{K}}$ and we compute some linearization formulas that play a role in the proof of Theorem \ref{thm:a}.
In Section \ref{sect:main}, we prove our main results.
Finally, in Appendix \ref{appendixA} we collect some facts and detailed computations concerning the Chern connection and the Weyl connection of a locally conformally K\"ahler manifold.
\bigskip

\noindent {\itshape Acknowledgements.}
The authors are warmly grateful to Vestislav Apostolov, Giovanni Bazzoni, Abdellah Lahdili, Mehdi Lejmi and Liviu Ornea for several interesting discussions and suggestions on the topic. We also thank the anonymous referees for their careful reading of the manuscript and for their suggestions.

\medskip
\section{Preliminaries and notation} \label{sect:prel}
\setcounter{equation} 0

Let $M^{2n}$ be a compact, connected, orientable, smooth manifold of real dimension $2n$. For any vector bundle $E \to M$ over $M$, we denote by $\mathcal{C}^{\infty}(M;E)$ the space of smooth sections of $E$. For the sake of notation, we set $\mathfrak{X}(M) \coloneqq \mathcal{C}^{\infty}(M;TM)$ and, for any $X \in \mathfrak{X}(M)$, $\Theta^X_t : M \to M$ will be the flow of $X$, $t \in \mathbb{R}$.

\subsection{Locally conformally symplectic manifolds} \label{sect:prel1} \hfill \par

We fix an {\it almost symplectic structure on $M$}, {\it i.e.\ }a smooth $2$-form $\omega$ such that $\omega_x$ is non degenerate for any $x \in M$, that corresponds to an ${\rm Sp}(2n,\mathbb{R})$-reduction of the frame bundle of $M$. We say that $\omega$ is {\it locally conformally symplectic} if, for any point $x \in M$, there exist a neighborhood $U \subset M$ of $x$ and a smooth function $f: U_x \to \mathbb{R}$ such that $e^{-f}\omega|_{U}$ is closed. By the Poincar\'e Lemma, the local conformal changes are organized in a $1$-form $\theta$, called the {\it Lee form of $(M,\omega)$}, that verifies the condition 
\begin{equation} \label{def:lcs}
{\rm d}\omega=\theta\wedge\omega \,\, , \quad {\rm d}\theta = 0 \,\, .
\end{equation}
By introducing the {\it twisted exterior differential operator}
\begin{equation} \label{def:dtwisted}
{\rm d}_\theta \coloneqq {\rm d} -\theta\wedge \,\, ,
\end{equation}
from \eqref{def:lcs} it follows that ${\rm d}_\theta\omega=0$. Moreover, as $\theta$ is closed, it is immediate to check that ${\rm d}_\theta \circ {\rm d}_\theta = 0$ and so ${\rm d}_\theta$ defines a cohomology $H^*_{{\rm d}_\theta}(M,\mathbb{R})$ called {\it Morse-Novikov cohomology}. For later use, we prove the following

\begin{lemma}
For any vector fields $X,Y \in \mathfrak{X}(M)$, the following equalities hold true:
\begin{gather}
\mathcal{L}_X\omega = {\rm d}_\theta(X \lrcorner \omega) +\theta(X)\omega \,\, , \label{eq:Lomega} \\
{\rm d}_\theta(\omega(X,Y)) = X \lrcorner ({\rm d}_\theta(Y \lrcorner \omega)) -Y \lrcorner ({\rm d}_\theta(X \lrcorner \omega)) -[X,Y] \lrcorner\omega \,\, ,\label{eq:commutator}
\end{gather}
where $\lrcorner$ denotes the interior product $X \lrcorner \omega \coloneqq \omega(X,\_)$.
\end{lemma}

\begin{proof}
Equation \eqref{eq:Lomega} is a direct application of the Cartan formula and the locally conformally symplectic condition ${\rm d}_{\theta}\omega=0$. Moreover, by \cite[Proposition I.3.10(b)]{KN1} and \eqref{eq:Lomega} we get
\begin{align*}
[X,Y] \lrcorner\omega &= [\mathcal{L}_X,Y\lrcorner](\omega) \\
&= {\rm d} (X \lrcorner (Y \lrcorner \omega)) +X\lrcorner({\rm d}(Y \lrcorner \omega)) -Y \lrcorner({\rm d}_\theta(X \lrcorner \omega)) -\theta(X)(Y \lrcorner\omega) \\
&= -{\rm d}_\theta(\omega(X,Y)) -\omega(X,Y)\theta +X \lrcorner ({\rm d}_\theta(Y \lrcorner \omega)) +X \lrcorner (\theta \wedge(Y \lrcorner \omega)) -Y \lrcorner({\rm d}_\theta(X \lrcorner \omega)) -\theta(X)(Y \lrcorner\omega) \\
&= -{\rm d}_\theta(\omega(X,Y)) +X \lrcorner ({\rm d}_\theta(Y \lrcorner \omega))  -Y \lrcorner({\rm d}_\theta(X \lrcorner \omega)) \,\, ,
\end{align*}
which concludes the proof.
\end{proof}

Symplectic structures are recovered for $\theta=0$. More generally, if $\theta$ is exact, then $\omega$ admits a global conformal change to a symplectic structure: in this case, we say that $\omega$ is {\it globally conformally symplectic}. In this paper, we will focus on {\it stricly locally conformally symplectic structures}, namely, the case when $\theta$ is non-exact. We refer to {\it e.g.\ }\cite{dragomir-ornea, bazzoni, bgp, ornea-verbitsky} and the references therein for an up-to-date account on locally conformally symplectic geometry. In the stricly locally conformally symplectic case, by \cite[Proposition 2.1]{vaisman85} we have the following

\begin{lemma} \label{lem:cohom0}
If $\theta$ is non-exact, then $H^0_{{\rm d}_\theta}(M,\mathbb{R}) = \{0\}$.
\end{lemma}

We recall (see {\it e.g.\ }\cite[Section 2]{gini-ornea-parton-piccinni} and \cite[Section 2.1]{bgp}) that the locally conformally symplectic manifold $(M,\omega)$ admits a unique, up to equivalence, {\it minimal symplectic covering}, that is the data of: \begin{itemize}
\item[$\bcdot$] a covering $\pi: \widehat{M} \to M$ with deck transformation group $\Gamma$;
\item[$\bcdot$] an injective group homomorphism $\rho: \Gamma \to (\mathbb{R},+)$ and a smooth function $f: \widehat{M} \to \mathbb{R}$ verifying
$$
f \circ \gamma - f = \rho(\gamma) \quad \text{ for any $\gamma \in \Gamma$} \,\, ;
$$
\item[$\bcdot$] a symplectic form $\widehat{\omega}_o$ on $\widehat{M}$ such that $\pi^*\omega = e^f\,\widehat{\omega}_o$.
\end{itemize}
In the following, we will denote the minimal symplectic covering of $(M,\omega)$ simply by $(\widehat{M},\widehat{\omega}_o)$. \smallskip

We recall that an {\it almost complex structure on $M$} is a section $J \in \mathcal{C}^{\infty}(M;{\rm End}(TM))$ such that $J^2 = -{\rm Id}$. Notice that $J$ acts on $T^*M$ by $J\varphi \coloneqq \varphi \circ J^{-1}$ and therefore it extends to the whole tensor bundle over $M$. By the Newlander-Nirenberg Theorem, it is known that $J$ is integrable, {\it i.e.\ }$M^{2n}$ admits a complex manifold structure that induces $J$, if and only if the Nijenhuis tensor of $J$ vanishes, {\it i.e.\ }
$$
N_J(X,Y) \coloneqq [JX,JY] -[X,Y] -J[JX,Y] -J[X,JY] = 0 \,\, .
$$
For the sake of notation, any integrable almost complex structure will just be called {\it complex structure}. An (almost) complex structure $J$ on $(M,\omega)$ is said to be {\it compatible with $\omega$} if
\begin{equation} \label{def:compatibilty}
\omega(J\_,J\_)=\omega \quad \text{ and } \quad \omega(\_,J\_)_x>0 \quad \text{ for any $x\in M$. }
\end{equation}
Any such (almost) complex structure induces a Riemannian metric $g_J \coloneqq \omega(\_,J\_)$ which is said to be {\it locally conformally (almost) K\"ahler}. Indeed, the pulled back metric $\pi^*g_J$ to the minimal symplectic covering of $(M,\omega)$ is globally conformal to the (almost) K\"ahler metric $\widehat{\omega}_o(\_,\pi^*J\_)$. We also remark that the Riemannian volume form on $M$ induced by $g_J$ coincides with the top exterior power $\omega^n$.

\subsection{Locally conformally K\"ahler metrics} \hfill \par

Let us fix a compatible complex structure $J$ on $(M,\omega)$ and denote by $g \coloneqq g_J$ the corresponding locally conformally K\"ahler metric. In the following, we will denote by $\_^{\sharp}$ and $\_^{\flat}$ the musical isomorphisms induced by $g$, {\it e.g.\ }
$$\begin{array}{ccccc}
q \in \mathcal{C}^{\infty}(M;S^2(TM)) &\mapsto &q^{\sharp} \in \mathcal{C}^{\infty}(M;{\rm Sym}(TM,g)) &\text{such that} &q(X,Y) = g(q^{\sharp}(X),Y) \,\, , \\[2pt]
\varphi \in \mathcal{C}^{\infty}(M;T^*M) &\mapsto &\varphi^{\sharp} \in \mathcal{C}^{\infty}(M;TM) &\text{such that} &\varphi(X) = g(\varphi^{\sharp},X) \,\, .
\end{array}$$
Moreover, given two vector bundles of tensors $E, F \to M$ and a linear differential operator $P: \mathcal{C}^{\infty}(M;E) \to \mathcal{C}^{\infty}(M;F)$, we denote by $P^*: \mathcal{C}^{\infty}(M;F) \to \mathcal{C}^{\infty}(M;E)$ its adjoint with respect to the $L^2$-pairing $\langle \_, \_ \rangle_{L^2(M,g)}$ induced by $g$. We denote by $D$ the {\it Levi-Civita connection} of $(M,g)$ and by
\begin{equation} \label{def:Riemcurv}
{\rm Rm}(X,Y) \coloneqq D_{[X,Y]}-[D_X,D_Y] \,\, , \quad {\rm Ric}(X,Y) \coloneqq {\rm tr}\, ( Z \mapsto {\rm Rm}(X,Z)Y )  \,\, , \quad {\rm scal} \coloneqq {\rm tr}\,({\rm Ric})^{\sharp}
\end{equation}
the {\it Riemannian curvature tensor}, the {\it Riemannian Ricci tensor} and the {\it Riemannian scalar curvature} of $(M,g)$, respectively. Furthermore, we denote by $\delta$ the {\it Riemannian divergence} acting on symmetric endomorphism fields of $(M,g)$, that can be locally written as \cite[Section 1.59]{besse}
$$
\delta \colon \mathcal{C}^{\infty}(M; {\rm Sym}(TM,g)) \to \mathcal{C}^{\infty}(M;T^*M) \,\, , \quad (\delta u)(X) \coloneqq -\sum_{\alpha=1}^{2n} g((D_{\tilde{e}_{\alpha}}u)(X),\tilde{e}_{\alpha}) \,\, ,
$$
where $(\tilde{e}_{\alpha})_{\alpha \in \{1,\ldots,2n\}}$ is any local $g$-orthonormal frame on $M$. 

The Levi-Civita connections of the local K\"ahler metrics induced by $g$ glue together to a globally defined connection $\nabla$, called the {\it Weyl connection} of $(M,J,g)$ \cite{tallini, vaisman76}. By \cite[Equation (2.4)]{vaisman76}, it has the following expression:
\begin{equation} \label{def:Weyl}
\nabla_X Y = D_XY - \tfrac12 A(X,Y) \,\, , \quad \text{ with } \quad A(X,Y) \coloneqq \theta(X)Y+\theta(Y)X-g(X,Y)\theta^{\sharp} \,\, .
\end{equation}
By the very definition and direct computations, it satisfies \cite[Theorem 2.2]{vaisman76}
\begin{equation} \label{eq:propertiesweyl}
T^\nabla=0 \,\, , \quad \nabla \omega =\theta\otimes \omega \,\, , \quad \nabla J=0 \,\, , \quad \nabla g=\theta\otimes g \,\, ,
\end{equation}
where $T^\nabla(X,Y) \coloneqq \nabla_XY- \nabla_YX - [X,Y]$ denotes the torsion tensor of $\nabla$. For later use, we prove the following

\begin{lemma}
Let $X \in \mathfrak{X}(M)$ be a vector field and $\zeta \in \mathcal{C}^{\infty}(M,T^*M)$ a 1-form. Then
\begin{gather} 
{\rm d}_{\theta}(X \lrcorner \omega) = \omega(\nabla_{\_} X,\_) +\omega(\_,\nabla_{\_} X) \,\, , \label{eq:holvf2} \\
\delta^*\zeta = \nabla(\zeta^{\sharp}) -\tfrac12(\_\, \lrcorner {\rm d}_\theta \zeta)^{\sharp} +\tfrac12g(\zeta,\theta){\rm Id} \,\, , \label{eq:delta*}
\end{gather}
where ${\rm d}_\theta$ is the twisted differential operator defined in \eqref{def:dtwisted}.
\end{lemma}

\begin{proof}
By \eqref{def:dtwisted} and \eqref{eq:propertiesweyl}, we get
\begin{align*}
{\rm d}_{\theta}(X \lrcorner \omega)(Y,Z) &= {\rm d}(X \lrcorner \omega)(Y,Z) -(\theta \wedge (X \lrcorner \omega))(Y,Z) \\
&= \mathcal{L}_Y\big(\omega(X,Z)\big) -\mathcal{L}_Z\big(\omega(X,Y)\big) -\omega(X,[Y,Z]) -\theta(Y)\omega(X,Z) +\theta(Z)\omega(X,Y) \\
&= \omega(\nabla_YX,Z) +\omega(Y,\nabla_ZX)
\end{align*}
and so \eqref{eq:holvf2} follows. Fix now a local $g$-orthonormal frame $(\tilde{e}_{\alpha})_{\alpha \in \{1,\ldots,2n\}}$ on $M$. By \cite[Lemma 1.60]{besse} (see also \cite[Errata at page 514]{besse}), we get
\begin{align*}
(\delta^*\zeta)(Y) &= \tfrac12\sum_{\alpha=1}^{2n} \big( g(D_Y(\zeta^{\sharp}),\tilde{e}_{\alpha}) +g(D_{\tilde{e}_{\alpha}}(\zeta^{\sharp}),Y)\big) \tilde{e}_{\alpha} \\
&= \tfrac12\sum_{\alpha=1}^{2n} \big( (D_Y\zeta)(\tilde{e}_{\alpha}) +\mathcal{L}_{\tilde{e}_{\alpha}}(\zeta(Y)) -\zeta(D_{\tilde{e}_{\alpha}}Y) \big) \tilde{e}_{\alpha} \\
&= \tfrac12\sum_{\alpha=1}^{2n} \big( (D_Y\zeta)(\tilde{e}_{\alpha}) +\mathcal{L}_{Y}(\zeta(\tilde{e}_{\alpha})) +\zeta([\tilde{e}_{\alpha},Y]) +{\rm d}\zeta(\tilde{e}_{\alpha},Y) -\zeta(D_{\tilde{e}_{\alpha}}Y) \big) \tilde{e}_{\alpha} \\
&= \tfrac12\sum_{\alpha=1}^{2n} \big( 2(D_Y\zeta)(\tilde{e}_{\alpha}) -{\rm d}\zeta(Y,\tilde{e}_{\alpha}) \big) \tilde{e}_{\alpha} \\
&= D_Y(\zeta^{\sharp}) -\tfrac12(Y\lrcorner {\rm d}\zeta)^{\sharp}
\end{align*}
and so, by \eqref{def:dtwisted} and \eqref{def:Weyl}, we get
\begin{align*}
(\delta^*\zeta)(Y) &= \nabla_Y(\zeta^{\sharp}) +\tfrac12\big(\theta(Y)\zeta^{\sharp} + \theta(\zeta^{\sharp})Y - \zeta(Y)\theta^{\sharp} \big) -\tfrac12(Y\lrcorner {\rm d}_{\theta}\zeta)^{\sharp} -\tfrac12\big(\theta(Y)\zeta^{\sharp} -\zeta(Y)\theta^{\sharp} \big) \\
&= \nabla_Y(\zeta^{\sharp}) -\tfrac12(Y\lrcorner {\rm d}_{\theta}\zeta)^{\sharp}
+\tfrac12g(\zeta,\theta)Y \,\, ,
\end{align*}
which concludes the proof of \eqref{eq:delta*}.
\end{proof}

We also denote by $\nabla^{\rm Ch}$ and ${\rm scal}^{\rm Ch}$ the {\it Chern connection} and the {\it Chern-scalar curvature} of $(M,J,g)$, respectively. For the convenience of the reader, we collected in Appendix \ref{appendixA} some properties of both the connections $\nabla$ and $\nabla^{\rm Ch}$ that will play a role afterwards. \smallskip

Finally, we recall that a (real) smooth vector field $X \in \mathfrak{X}(M)$ is said to be {\it $J$-holomorphic} if $\mathcal{L}_XJ=0$. Notice that, as $J$ is integrable, by \eqref{eq:propertiesweyl} we directly get

\begin{lemma} \label{lem:holvf1}
For any $X \in \mathfrak{X}(M)$ we have
\begin{equation}
\mathcal{L}_XJ = [J,\nabla X] \,\, , \quad \mathcal{L}_{JX}J = J \circ \mathcal{L}_XJ \,\, , \quad \mathcal{L}_XJ \circ J + J \circ \mathcal{L}_XJ = 0 \,\, . \label{eq:holvf1} \end{equation}
In particular, the following three conditions are equivalent: \begin{itemize}
\item[$\bcdot$] $X$ is $J$-holomorphic,
\item[$\bcdot$] $JX$ is $J$-holomorphic,
\item[$\bcdot$] $[J,\nabla X] = 0$.
\end{itemize}
\end{lemma}

\medskip
\section{Twisted Hamiltonian diffeomorphisms and torus actions} \label{sect:twH-torus}
\setcounter{equation} 0

Let $(M^{2n},\omega)$ be a compact, connected, smooth, locally conformally symplectic manifold of real dimension $2n$ with non-exact Lee form $\theta$. We denote by $(\widehat{M},\widehat{\omega}_o)$ its minimal symplectic cover and by $\Gamma$ the deck transformation group of the projection $\pi: \widehat{M} \to M$. For the sake of notation, we denote by $\mathfrak{X}(\widehat{M})^{\Gamma}$ the closed subspace of $\Gamma$-invariant, smooth vector fields on the total space $\widehat{M}$ and by
\begin{equation} \label{def:simplecticLiealg} \begin{gathered}
\mathfrak{aut}(\widehat{M},\widehat{\omega}_o)^{\Gamma} \coloneqq \big\{ X \in \mathfrak{X}(\widehat{M})^{\Gamma} : \mathcal{L}_{X}\widehat{\omega}_o = 0 \big\} = \big\{ X \in \mathfrak{X}(\widehat{M})^{\Gamma} : {\rm d}(X \lrcorner \widehat{\omega}_o) = 0 \big\} \,\, ,\\
\mathfrak{ham}(\widehat{M},\widehat{\omega}_o)^{\Gamma} \coloneqq \big\{ X \in \mathfrak{X}(\widehat{M})^{\Gamma} : \text{$X \lrcorner \widehat{\omega}_o$ is ${\rm d}$-exact} \big\} 
\end{gathered} \end{equation}
the Lie algebra of $\Gamma$-invariant infinitesimal automorphisms of $(\widehat{M},\widehat{\omega}_o)$ and the Lie algebra of $\Gamma$-invariant Hamiltonian vector fields of $(\widehat{M},\widehat{\omega}_o)$, respectively.

\subsection{A remark on infinite dimensional transformation groups} \label{sect:defG1} \hfill \par

We denote by ${\rm Diff}(M)$ the smooth diffeomorphism group of $M$ endowed with the compact-open $\mathcal{C}^{\infty}$-topology \cite[Section 2.1]{Hirsch} and we recall that it has a natural structure of infinite dimensional {\it strong ILH-Lie group} in the sense of Omori \cite[Definition III.3.1]{omori-TranslAMS}, with Lie algebra isomorphic to $\mathfrak{X}(M)$ \cite[Theorem 2.1.5]{omori-LNM}.
Accordingly, we denote by
\begin{equation} \label{def:expDiff}
\exp: \mathfrak{X}(M) \to {\rm Diff}(M) \,\, , \quad \exp(X) \coloneqq \Theta^X_1
\end{equation}
its exponential map and by
\begin{equation} \label{def:AdDiff}
{\rm Ad}: {\rm Diff}(M) \times \mathfrak{X}(M) \to \mathfrak{X}(M)  \,\, , \quad ({\rm Ad}(\varphi)X)_x \coloneqq (\varphi_*X)_x = {\rm d}\varphi\,|_{\varphi^{-1}(x)}(X_{\varphi^{-1}(x)})
\end{equation}
its adjoint representation.
Here and in the following, we will not go into details of the theory of infinite-dimensional Lie groups. As a good reference, we recommend \cite{omori-LNM, omori-TranslAMS}.

We also recall that an {\it isotopy to the identity} is a smooth map
$$
\phi: [0,1] \times M \to M \quad \text{ such that } \quad \phi_t \coloneqq \phi(t,\_) \in {\rm Diff}(M) \text{ for any } t \in [0,1] \,\, , \quad \phi_0 = {\rm Id}
$$
and that $\varphi \in {\rm Diff}(M)$ is said to be {\it isotopic to the identity} if there exists an isotopy to the identity $(\phi_t)_{t \in [0,1]}$ such that $\phi_1 = \varphi$. Since ${\rm Diff}(M)$ is locally connected by smooth arcs, it follows that its identity component ${\rm Diff}(M)_0$ coincides with the subset of diffeomorphisms of $M$ that are isotopic to the identity (see {\it e.g.\ }\cite[Section I.1]{banyaga}). We also recall that there exists a bijective correspondence between the set of all the isotopies to the identity and the set of {\it time-dependent vector fields} on $M$, {\it i.e.\ }the smooth maps $Z: [0,1] \times M \to TM$ such that $Z_t|_x \coloneqq Z(t,x) \in T_xM$ for any $t \in [0,1]$, $x \in M$ (see {\it e.g.\ }\cite[Theorem 9.48, Exercise 9-20, Exercise 9-21]{Leebook}). This relation is explicitly given by the following system of ODE's:
\begin{equation} \label{eq:diffisotopy}
\bigg\{\!\! \begin{array}{l}
\tfrac{\rm d}{{\rm d} t} \phi_t = Z_t \circ \phi_t \\
\phi_0 = {\rm Id}
\end{array} \,\, .
\end{equation}

Let now $\mathfrak{g} \subset \mathfrak{X}(M)$ be a closed Lie subalgebra of vector fields on $M$. In general, it is not possible to find a strong ILH-Lie subgroup in the sense of Omori whose Lie algebra is $\mathfrak{g}$ (see \cite[Section III.5]{omori-TranslAMS}). However, this problem of integrability can be overcome by considering the weaker notion of {\it generalized Lie group} in the sense of Omori \cite[Definition I.3.1, Definition I.3.2]{omori-TranslAMS} in the following way. First of all, we call {\it $\mathfrak{g}$-isotopy to the identity} any isotopy to the identity $(\phi_t)_{t \in [0,1]}$ such that the corresponding time-dependent vector field $(Z_t)_{t \in [0,1]}$ verifies $Z_t \in \mathfrak{g}$ for any $t \in [0,1]$. Then, the following result holds true.

\begin{proposition} \label{prop:integralgroup}
Let $\mathfrak{g} \subset \mathfrak{X}(M)$ be a closed Lie subalgebra of vector fields on $M$ and assume that
\begin{equation} \label{eq:g-isotopy} \tag{$\triangle$}
{\rm Ad}(\phi_t)X \subset \mathfrak{g} \quad \text{ for any $\mathfrak{g}$-isotopy to the identity $(\phi_t)_{t \in [0,1]}$ }, \,\, \text{ for any $X \in \mathfrak{g}$ }.
\end{equation}
Then, the subset
\begin{equation} \label{def:G}
G \coloneqq \{\varphi \in {\rm Diff}(M) : \text{there exists a $\mathfrak{g}$-isotopy to the identity $(\phi_t)_{t \in [0,1]}$ such that $\phi_1 = \varphi$} \} \,\, .
\end{equation}
is a generalized Lie group in the sense of Omori, with Lie algebra $\mathfrak{g}$, whose exponential map and adjoint representation are given by the restrictions of \eqref{def:expDiff} and \eqref{def:AdDiff}, respectively.
\end{proposition}

\begin{proof}
Let $\varphi^{(1)}, \varphi^{(2)} \in G$, take two $\mathfrak{g}$-isotopy to the identity $(\phi^{(i)}_t)_{t \in [0,1]}$ such that $\phi^{(i)}_1 = \varphi^{(i)}$ and denote by $(Z^{(i)})_{t \in [0,1]}$ the corresponding time-dependent vector field, with $i =1,2$. Then, it is immediate to observe that $(\phi_t)_{t \in [0,1]}$, with $\phi_t \coloneqq \phi^{(1)}_t \circ \phi^{(2)}_t$ for any $t \in [0,1]$, is an isotopy to the identity and that the corresponding time-dependent vector field $(Z_t)_{t \in [0,1]}$ is given by
\begin{equation} \label{eq:compisotopies}
Z_t = Z^{(1)}_t +{\rm Ad}(\phi^{(1)}_t)Z^{(2)}_t \,\, , \quad t \in [0,1] \,\, .
\end{equation}
By \eqref{eq:g-isotopy} and \eqref{eq:compisotopies}, it follows that $(\phi_t)_{t \in [0,1]}$ is a $\mathfrak{g}$-isotopy to the identity and so, since $\phi_1 = \varphi^{(1)} \circ \varphi^{(2)}$, it follows that $G$ is a subgroup of ${\rm Diff}(M)_0$. It is straightforward now to check that $G$ verifies all the properties listed in
\cite[Definition I.3.1, Definition I.3.2]{omori-TranslAMS}.
\end{proof}

According to Proposition \ref{prop:integralgroup}, we refer to the group $G$ defined in \eqref{def:G} as the {\it generalized Lie transformation group generated by $\mathfrak{g}$}.

\subsection{The twisted-Hamiltonian diffeomorphism group} \label{sect:defG2} \hfill \par

Following \cite{mamopi17, bgp}, we consider the set of {\it special conformal vector fields} of $(M,\omega)$
\begin{equation} \label{def:specconfaut}
\mathfrak{aut}^{\star}(M,[\omega]) \coloneqq \big\{ X \in \mathfrak{X}(M) : \mathcal{L}_X\omega = \theta(X)\omega \big\} \overset{\eqref{eq:Lomega}}{=} \big\{ X \in \mathfrak{X}(M) : {\rm d}_{\theta}(X \lrcorner \omega) = 0 \big\}
\end{equation}
and the subset of {\it twisted-Hamiltonian vector fields of $(M,\omega)$}
\begin{equation} \label{def:twham}
\mathfrak{ham}(M,[\omega]) \coloneqq \big\{ X \in \mathfrak{X}(M) : X \lrcorner \omega = {\rm d}_{\theta}h \text{ for some } h \in \mathcal{C}^{\infty}(M,\mathbb{R}) \big\} \,\, .
\end{equation}
By \eqref{eq:commutator}, it is easy to observe that both \eqref{def:specconfaut} and \eqref{def:twham} are closed Lie subalgebras of $\mathfrak{X}(M)$ and that 
$$
\big[\mathfrak{aut}^{\star}(M,[\omega]),\mathfrak{aut}^{\star}(M,[\omega])\big] \subset \mathfrak{ham}(M,[\omega]) \subset \mathfrak{aut}^{\star}(M,[\omega]) \,\, .
$$
Since $\omega$ is non degenerate, from Lemma \ref{lem:cohom0} it follows that there exists a unique linear isomorphism
\begin{equation} \label{eq:kappa}
\kappa: \mathfrak{ham}(M,[\omega]) \to \mathcal{C}^{\infty}(M,\mathbb{R}) \quad \text{ such that } \quad X \lrcorner \omega = {\rm d}_\theta (\kappa(X)) \,\, \text{ for any } X \in \mathfrak{ham}(M,[\omega]) \,\, .
\end{equation}
Moreover, by \eqref{eq:commutator}, the isomorphism $\kappa$ verifies
$$
\kappa([X,Y]) = -\omega(X,Y) \quad \text{ for any } X, Y \in \mathfrak{aut}^{\star}(M,[\omega]) \,\, .
$$
It is remarkable to notice that, by \cite{bgp}, the Lie algebras \eqref{def:specconfaut} and \eqref{def:twham} are related to the Lie algebras defined in \eqref{def:simplecticLiealg} by the following

\begin{proposition}[{\cite[Proposition 3.3, Corollary 3.10]{bgp}}] \label{prop:bgp}
The covering map $\pi: \widehat{M} \to M$ induces the following two Lie algebra isomorphisms:
$$
\pi_* : \mathfrak{aut}(\widehat{M},\widehat{\omega}_o)^{\Gamma} \to \mathfrak{aut}^{\star}(M,[\omega]) \,\, , \quad \pi_* : \mathfrak{ham}(\widehat{M},\widehat{\omega}_o)^{\Gamma} \to \mathfrak{ham}(M,[\omega]) \,\, ,
$$
where $\Gamma$ is the deck transformation group of $\pi$.
\end{proposition}

As a direct corollary of Proposition \ref{prop:bgp}, by using a lifting argument, one can prove the following

\begin{corollary} \label{cor:bgp}
Both the Lie algebras $\mathfrak{aut}^{\star}(M,[\omega])$ and $\mathfrak{ham}(M,[\omega])$ verify the property \eqref{eq:g-isotopy}.
\end{corollary}

According to Proposition \ref{prop:integralgroup} and Corollary \ref{cor:bgp}, we denote by ${\rm Aut}^{\star}(M,[\omega])$ and ${\rm Ham}(M,[\omega])$ the ge\-ne\-ra\-lized Lie transformation groups generated by $\mathfrak{aut}^{\star}(M,[\omega])$ and $\mathfrak{ham}(M,[\omega])$, respectively. We call them {\it special conformal automorphism group} and {\it twisted-Hamiltonian diffeomorphism group of $(M,\omega)$}, respectively.

\subsection{A note on twisted-Hamiltonian group actions} \label{sect:actionK} \hfill \par

We denote by $V \in \mathfrak{X}(M)$ the {\it symplectic dual of $\theta$}, {\it i.e.\ }the unique vector field on $M$ defined by the condition $V \lrcorner \omega = \theta$. Notice that from \eqref{def:dtwisted} and the very definition, we get
\begin{equation} \label{eq:Lham}
V \lrcorner \omega = {\rm d}_{\theta}(-1) \,\, , \quad \theta(V)=0
\end{equation}
and therefore, by \eqref{eq:Lomega}, it follows that
\begin{equation} \label{eq:LVomega}
\mathcal{L}_V\omega = 0 \,\, .
\end{equation}

Let now $\mathsf{K}$ be a compact, connected Lie group that acts effectively on $M$ and $\mathfrak{k} \coloneqq {\rm Lie}(\mathsf{K}) \subset \mathfrak{X}(M)$ its Lie algebra. By compactness, $\mathsf{K}$ can be written as the product $\mathsf{K} = \mathsf{Z}(\mathsf{K}) \cdot [\mathsf{K},\mathsf{K}]$, where $\mathsf{Z}(\mathsf{K})$ and $[\mathsf{K},\mathsf{K}]$ denote the center and the commutator of $\mathsf{K}$, respectively. At the Lie algebra level, this decomposition corresponds to $\mathfrak{k} = \mathfrak{z}(\mathfrak{k}) \oplus [\mathfrak{k},\mathfrak{k}]$. We make the following assumptions on the group $\mathsf{K}$:
\begin{enumerate}[label={\rm (k1)}, leftmargin=30pt]
\item\label{(k1)} the $\mathsf{K}$-action preserves $\omega$;
\end{enumerate} \begin{enumerate}[label={\rm (k2)}, leftmargin=30pt]
\item\label{(k2)} the $\mathsf{K}$-action is of {\it Lee type}, {\it i.e.\ }$V \in \mathfrak{z}(\mathfrak{k})$;
\end{enumerate} \begin{enumerate}[label={\rm (k3)}, leftmargin=30pt]
\item\label{(k3)} the $\mathsf{K}$-action is {\it twisted-Hamiltonian}, {\it i.e.\ }$\mathsf{K} \subset {\rm Ham}(M,[\omega])$.
\end{enumerate}
Let us notice that condition \ref{(k1)} is not too restrictive, in view of the following

\begin{lemma}
Let $\mathsf{K}$ be a compact, connected Lie group acting effectively on $M$. If $M$ admits a locally conformally symplectic form, then it also admits a $\mathsf{K}$-invariant locally conformally symplectic form.
\end{lemma}

\begin{proof}
Fix $\mathsf{K}$ a locally conformally symplectic form on $M$ with Lee form $\theta$. Since the cohomology of $M$ is $\mathsf{K}$-invariant, see {\it e.g.} \cite[Proposition 1.28]{felix-oprea-tanre}, we can perform a conformal change in order to make $\theta$ $\mathsf{K}$-invariant. Then, the statement follows by standard averaging methods.
\end{proof}

Notice that, by condition \ref{(k2)}, the group $\mathsf{K}$ cannot be semisimple. For example, the standard actions of $\mathsf{SU}(2)$ on the Hirzebruch surfaces are not of Lee-type. A less trivial example of a non-Lee-type action is given by the homogenous action of $\mathsf{U}(2)$ on the linear Hopf manifold. Indeed one can directly check that the vector $V$ turns out to be tangent to the semisimple part of $\mathsf{U}(2)$, see \cite[Prop. 3.17]{angella-pediconi}.

\begin{remark} \label{rem:Lclosure}
Let $\mathsf{L}$ be the flow of $V$, {\it i.e.\ }the 1-parameter subgroup of ${\rm Diff}(M)$ defined as
\begin{equation}\label{eq:def-L}
\mathsf{L} \coloneqq \{\Theta^V_t : t \in \mathbb{R} \} \,\, .
\end{equation}
We notice that $\mathsf{L} \subset {\rm Ham}(M)$ thanks to \eqref{eq:Lham}.
Consider now its closure $\overline{\mathsf{L}}$ in ${\rm Diff}(M)$ with respect to the compact-open topology. Notice that, if there exists an almost complex structure $J$ compatible with $\omega$ such that
\begin{equation}\label{eq:LVJ}
\mathcal{L}_V J=0 \,\, ,
\end{equation}
then, by \eqref{eq:LVomega}, $V$ is a Killing vector field for the Riemannian metric $g_{J} = \omega(\_,J\_)$. Therefore, by the Myers-Steenrod Theorem, it follows that $\overline{\mathsf{L}} = \mathsf{T}$ is a torus, whose action on $M$ clearly satisfies conditions \ref{(k1)} and \ref{(k2)}.
We stress that, if $(M,\omega)$ admits a compatible complex structure $\hat J$ whose associated Riemannian metric is Vaisman, then condition \eqref{eq:LVJ} is automatically satisfied and thus it admits a torus action of Lee type.
On the other hand, in \cite{moroianu-moroianu-ornea}, the authors construct examples of compact locally conformally K\"ahler manifolds
satisfying \eqref{eq:LVJ} which are not Vaisman.
\end{remark}

Concerning Remark \ref{rem:Lclosure} it is worth mentioning that, in general, even when $\mathsf{T}=\overline{\mathsf{L}}$ is compact, it is not obvious that condition \ref{(k3)} is satisfied, {\it i.e.} that $\mathsf{T}$ is contained in ${\rm Ham}(M,[\omega])$. Indeed, by means of Proposition \ref{prop:bgp}, it follows that the special conformal automorphism group ${\rm Aut}^{\star}(M,[\omega])$ is closed in ${\rm Diff}(M)$ and so, since ${\rm d}_{\theta}(V \lrcorner \omega) = 0$, we get $\mathfrak{t} \coloneqq {\rm Lie}(\mathsf{T}) \subset \mathfrak{aut}^{\star}(M,[\omega])$. On the other hand, it is not known, to the best of our knowledge, whether the twisted-Hamiltonian diffeomorphism group ${\rm Ham}(M,[\omega])$ is closed or not, and therefore we do not know whether the inclusion $\mathfrak{t} \subset \mathfrak{ham}(M,[\omega])$ holds in general or not. We remark that the corresponding problem in the classical setting has been solved in \cite{ono}, where the author proved, as a corollary of his main theorem, that the Hamiltonian diffeomorphism group of a compact, connected symplectic manifold is closed in the diffeomorphism group with the $\mathcal C^\infty$-topology.

\medskip
\section{\texorpdfstring{$\mathsf{K}$}{K}-invariant, compatible complex structures} \label{sect:Jomega}
\setcounter{equation} 0

Let $(M^{2n},\omega)$ be a compact, connected, smooth, locally conformally symplectic manifold of real dimension $2n$ with non-exact Lee form $\theta$. Assume that $\mathsf{K}$ is a compact, connected Lie group that acts effectively on $M$ and verifies \ref{(k1)}, \ref{(k2)}, \ref{(k3)} as in Section \ref{sect:actionK}. In the following, for any vector bundle of tensors $E \to M$ over $M$, we denote by $\mathcal{C}^{\infty}(M;E)^{\mathsf{K}}$ the closed subspace of smooth sections of $E$ that are $\mathsf{K}$-invariant.

\subsection{The space of \texorpdfstring{$\mathsf{K}$}{K}-invariant, compatible almost complex structures} \hfill \par

Let us define
\begin{equation} \label{def:JalmomegaT}
\mathcal{J}_{\rm alm}(\omega)^{\mathsf{K}} \coloneqq \big\{ \text{ $J$ almost complex structure on $M$ compatible with $\omega$ as in \eqref{def:compatibilty} and $\mathsf{K}$-invariant } \big\} \,\, .
\end{equation}
Following \cite[Remark 4.3]{fujiki}, the space $\mathcal{J}_{\rm alm}(\omega)^{\mathsf{K}}$ has a natural structure of {\it smooth ILH manifold} in the sense of Omori \cite[Definition I.1.9]{omori-TranslAMS}, since it can be regarded as the space of $\mathsf{K}$-invariant smooth sections of a fibre bundle with typical fibre the Siegel upper half-space, see also \cite[Theorem A]{koiso}. Note indeed that, in the classical picture for almost K\"ahler structures, one does only need $\omega$ to be non-degenerate.

\begin{remark}
By Remark \ref{rem:Lclosure}, it follows that if the closure $\overline{\mathsf{L}}$ of the 1-parameter group generated by $V$ is non-compact, then the space of $\mathsf{\mathsf{L}}$-invariant almost complex structure on $M$ compatible with $\omega$ is empty.
\end{remark}

We recall now the characterization of the tangent space of $\mathcal{J}_{\rm alm}(\omega)^{\mathsf{K}}$. In the following, we denote by $\mathfrak{sp}(TM,\omega) \subset {\rm End}(TM)$ the subbundle defined fiber-wise by
$$
\mathfrak{sp}(T_xM,\omega_x) \coloneqq \big\{ a \in {\rm End}(T_xM) : \omega_x(a\_,\_)+\omega_x(\_,a\_) = 0 \big\} \,\, , \quad x \in M \,\, .
$$
 
\begin{lemma}
The tangent space of $\mathcal{J}_{\rm alm}(\omega)^{\mathsf{K}}$ at $J$ is given by
\begin{equation} \label{eq:tangent-space} \begin{aligned}
T_J\mathcal{J}_{\rm alm}(\omega)^{\mathsf{K}} &= \big\{ u \in \mathcal{C}^{\infty}(M; \mathfrak{sp}(TM,\omega))^{\mathsf{K}} : Ju+uJ=0 \big\} \\
&= \big\{ \hat{a}_J \coloneqq [J,a] : a \in \mathcal{C}^{\infty}(M;\mathfrak{sp}(TM,\omega))^{\mathsf{K}} \big\} \,\, .
\end{aligned} \end{equation}
\end{lemma}

\begin{proof}
Let $(J_t)_t \subset \mathcal{J}_{\rm alm}(\omega)^{\mathsf{K}}$ be a curve starting at $J = J_0$ with initial tangent vector $u = J_0' \in T_J\mathcal{J}_{\rm alm}(\omega)^{\mathsf{K}}$. In particular, $J_t$ satisfies $J_t^2=-{\rm Id}$ and $\omega(J_t\_,J_t\_) = \omega$, and so
\begin{align*}
0 &= \tfrac{\rm d}{{\rm d}t} \,J_t^2\, \big|_{t=0} = J_0 J_0' +J_0' J_0 = Ju +uJ \,\, , \\
0 &= -\tfrac{\rm d}{{\rm d}t} \,\omega(J_t\_,J_tJ\_)\, \big|_{t=0} = \omega(u\_,\_)-\omega(J\_,uJ\_) = \omega(u\_,\_)+\omega(\_,u\_) \,\, ,
\end{align*}
which proves the first inclusion
$$
T_J\mathcal{J}_{\rm alm}(\omega)^{\mathsf{K}} \subset \big\{ u \in \mathcal{C}^{\infty}(M; \mathfrak{sp}(TM,\omega))^{\mathsf{K}} : Ju+uJ=0 \big\} \,\, .
$$
Notice now that any $u \in \mathcal{C}^{\infty}(M; \mathfrak{sp}(TM,\omega))$ satisfying $Ju+uJ=0$ can be written as
$$
u = \big[J, -\tfrac12Ju\big]
$$
with
$$
\omega(-\tfrac12Ju\_,\_)+\omega(\_,-\tfrac12Ju\_) = \tfrac12\omega(u\_,J\_) +\tfrac12\omega(\_,uJ\_) = 0 \,\, ,
$$
which proves the second inclusion
$$
\big\{ u \in \mathcal{C}^{\infty}(M; \mathfrak{sp}(TM,\omega))^{\mathsf{K}} : Ju+uJ=0 \big\} \subset \big\{ \hat{a}_J \coloneqq [J,a] : a \in \mathcal{C}^{\infty}(M;\mathfrak{sp}(TM,\omega))^{\mathsf{K}} \big\} \,\, .
$$
Finally, for any $a \in \mathcal{C}^{\infty}(M;\mathfrak{sp}(TM,\omega))^{\mathsf{K}}$, we define $J_t \coloneqq \exp(-ta) J \exp(ta)$ and we observe that $J_t \in \mathcal{J}_{\rm alm}(\omega)^{\mathsf{K}}$ for any $t \in \mathbb{R}$, $J_0 = J$ and $J_0' = [J,a] = \hat{a}_J$, which proves the remaining inclusion.
\end{proof}

As a consequence of \eqref{eq:tangent-space}, any $\mathsf{K}$-invariant tensor field $a \in \mathcal{C}^{\infty}(M;\mathfrak{sp}(TM,\omega))^{\mathsf{K}}$ determines a vector field on $\mathcal{J}_{\rm alm}(\omega)^{\mathsf{K}}$ by 
\begin{equation} \label{def:basicvf}
\hat{a} \in \mathcal{C}^{\infty}\big(\mathcal{J}_{\rm alm}(\omega)^{\mathsf{K}}; T\mathcal{J}_{\rm alm}(\omega)^{\mathsf{K}}\big) \,\, , \quad \hat{a}_J \coloneqq [J,a] \,\, .
\end{equation}
Any vector field of the form \eqref{def:basicvf} will be called {\it basic}. As in \cite[Equation 9.2.7]{gauduchon-book}, we prove the following

\begin{lemma}
The Lie bracket of two basic vector fields $\hat{a}, \hat{b}$ on $\mathcal{J}_{\rm alm}(\omega)^{\mathsf{K}}$ is basic and given by
\begin{equation} \label{eq:liebr}
[\hat{a},\hat{b}] = \widehat{[a,b]} \,\, , \quad \text{ with } \,\, [a,b] = ab-ba \,\, .
\end{equation}
\end{lemma}

\begin{proof}
Take $a, b \in \mathcal{C}^{\infty}(M;\mathfrak{sp}(TM,\omega))^{\mathsf{K}}$ and notice that the flow of the basic vector field $\hat{a}$ is
$$
\Theta^{\hat{a}}: \mathbb{R} \times \mathcal{J}_{\rm alm}(\omega)^{\mathsf{K}} \to \mathcal{J}_{\rm alm}(\omega)^{\mathsf{K}} \,\, , \quad \Theta^{\hat{a}}_t(J) = \exp(-ta) J \exp(ta) \,\, .
$$
Therefore, for any $J \in \mathcal{J}_{\rm alm}(\omega)^{\mathsf{K}}$ we get
\begin{align*}
[\hat{a},\hat{b}]_J &= \tfrac{\partial^2}{\partial t \, \partial s} \big(\Theta^{\hat{a}}_{-t} \circ \Theta^{\hat{b}}_{s} \circ \Theta^{\hat{a}}_{t}\big)(J) \, \big|_{(t,s)=(0,0)} \\
&= \tfrac{\partial^2}{\partial t \, \partial s} \exp(ta)\exp(-sb)\exp(-ta)J\exp(ta)\exp(sb)\exp(-ta) \, \big|_{(t,s)=(0,0)} \\
&= [J,ab-ba]
\end{align*}
which proves \eqref{eq:liebr}.
\end{proof}

Following \cite[page 179]{fujiki} (see also \cite[page 227]{gauduchon-book}), we define the {\it tautological complex structure $\mathbb{J}$} and the {\it tautological symplectic structure $\mathbb{\Omega}$} by setting
\begin{align}
\mathbb{J} \in \mathcal{C}^{\infty}\big(\mathcal{J}_{\rm alm}(\omega)^{\mathsf{K}}; {\rm End} (T\mathcal{J}_{\rm alm}(\omega)^{\mathsf{K}})\big) \,\, &, \quad \mathbb{J}_J \coloneqq J \,\, , \nonumber \\
\mathbb{G} \in \mathcal{C}^{\infty}\big(\mathcal{J}_{\rm alm}(\omega)^{\mathsf{K}};S^2_+(T^*\mathcal{J}_{\rm alm}(\omega)^{\mathsf{K}})\big) \,\, &, \quad \mathbb{G}_J(u,v) \coloneqq \langle u, v \rangle_{L^2(M,g_J)} = \int_M {\rm tr}(uv)\, \omega^n \,\, ,  \label{eq:tautK} \\
\mathbb{\Omega} \in \mathcal{C}^{\infty}\big(\mathcal{J}_{\rm alm}(\omega)^{\mathsf{K}};\Lambda^2(T^*\mathcal{J}_{\rm alm}(\omega)^{\mathsf{K}})\big) \,\, &, \quad \mathbb{\Omega}_J(u,v) \coloneqq \mathbb{G}_J(\mathbb J_Ju,v) = \int_M {\rm tr}(Juv)\, \omega^n \,\, . \nonumber
\end{align}
Here, $\langle \_, \_ \rangle_{L^2(M,g_J)}$ denotes the standard $L^2$-pairing  induced by $g_J$. We refer to the triple $(\mathbb{J}, \mathbb{G}, \mathbb{\Omega})$ as the {\it tautological K\"ahler structure of $\mathcal{J}_{\rm alm}(\omega)^{\mathsf{K}}$}. This nomenclature is due to the following result (compare with \cite[Theorem 4.2]{fujiki}).

\begin{proposition} \label{prop:almkahler}
The tautological K\"ahler structure $(\mathbb{J}, \mathbb{G}, \mathbb{\Omega})$ defined in \eqref{eq:tautK} is a K\"ahler structure on the ILH manifold $\mathcal{J}_{\rm alm}(\omega)^{\mathsf{K}}$ with respect to which the basic vector fields \eqref{def:basicvf} are holomorphic Killing.
\end{proposition}

\begin{proof}
By the very definitions, it is straightforward to realize that $(\mathbb{J}_J,\mathbb{G}_J)$ is a linear Hermitian structure on $T_J\mathcal{J}_{\rm alm}(\omega)^{\mathsf{K}}$ for any $J \in \mathcal{J}_{\rm alm}(\omega)^{\mathsf{K}}$. Since the flow $\Theta^{\mathbb{J}\hat{a}}$ of the vector field $\mathbb{J}\hat{a}$ is given by
$$
\Theta^{\mathbb{J}\hat{a}}_t(J) = \exp(-tJa) J \exp(tJa) \,\, ,
$$
a direct computation shows that for any $J \in \mathcal{J}_{\rm alm}(\omega)^{\mathsf{K}}$
\begin{align*}
[\hat{a},\mathbb{J}\hat{b}]_J &= \tfrac{\partial^2}{\partial t \, \partial s} \big(\Theta^{\hat{a}}_{-t} \circ \Theta^{\mathbb{J}\hat{b}}_{s} \circ \Theta^{\hat{a}}_{t}\big)(J) \, \big|_{(t,s)=(0,0)} \\
&= \tfrac{\partial^2}{\partial t \, \partial s} \exp(ta)\exp\!\big({-}s\exp({-}ta)J\exp(ta)b\big)\exp({-}ta) \cdot J \cdot \\
&\qquad\qquad\qquad\qquad \exp(ta)\exp\!\big(s\exp({-}ta)J\exp(ta)b\big)\exp({-}ta) \, \big|_{(t,s)=(0,0)} \\
&= \tfrac{\partial}{\partial t} \big( {-}J\exp(ta)b\exp({-}ta)J {-}\exp(ta)b\exp({-}ta) \big) \, \big|_{t=0} \\
&= -JabJ +JbaJ -ab +ba \\
&= J[J,[a,b]] \\
&= \big(\mathbb{J}[\hat{a},\hat{b}]\big)_J
\end{align*}
and so, by \eqref{eq:tangent-space}, it follows that
\begin{equation} \label{eq:almkahlerdim(1)}
\mathcal{L}_{\hat{a}}\mathbb{J} = 0 \,\, .
\end{equation}
An analogous computation shows that $\mathcal{L}_{\mathbb{J}\hat{a}}\mathbb{J} = 0$ and so
\begin{equation} \label{eq:almkahlerdim(2)}
N_{\mathbb{J}}(\hat{a},\hat{b}) \coloneqq
[\mathbb{J}\hat{a},\mathbb{J}\hat{b}] -[\hat{a},\hat{b}] -\mathbb{J}[\mathbb{J}\hat{a},\hat{b}] -\mathbb{J}[\hat{a},\mathbb{J}\hat{b}] = 0 \,\, .
\end{equation}
By \cite[Proposition 1.4]{koiso}, the integrability of $\mathbb{J}$ is equivalent to the vanishing of the Nijenhuis tensor $N_{\mathbb{J}}$ and so, by \eqref{eq:tangent-space} and \eqref{eq:almkahlerdim(2)}, it follows that $\mathbb{J}$ is integrable.

A straightforward computation shows that for any $J \in \mathcal{J}_{\rm alm}(\omega)^{\mathsf{K}}$
$$
{\rm tr}(J[J,a][J,b]) = 2{\rm tr}(J[a,b])
$$
and so
$$
\mathcal{L}_{\hat{a}}\big(\mathbb{\Omega}(\hat{b},\hat{c})\big)_J = \tfrac{\rm d}{{\rm d} t}\, \mathbb{\Omega}_{\Theta^{\hat{a}}_t(J)}\big(\hat{b}_{\Theta^{\hat{a}}_t(J)},\hat{c}_{\Theta^{\hat{a}}_t(J)}\big)\, \big|_{t=0} = 2 \frac{\rm d}{{\rm d} t} \int_M {\rm tr}(\Theta^{\hat{a}}_t(J)[b,c])\, \omega^n\, \Big|_{t=0} = 2 \int_M {\rm tr}(\hat{a}_J[b,c])\, \omega^n \,\, .
$$
Therefore
\begin{align*}
\big(\mathcal{L}_{\hat{a}}\mathbb{\Omega}\big)(\hat{b},\hat{c})_J &=
\mathcal{L}_{\hat{a}}\big(\mathbb{\Omega}(\hat{b},\hat{c})\big)_J
-\mathbb{\Omega}([\hat{a},\hat{b}],\hat{c})_J
-\mathbb{\Omega}(\hat{b},[\hat{a},\hat{c}])_J \\
&= 2 \int_M \big( {\rm tr}(\hat{a}_J[b,c]) -{\rm tr}(J[[a,b],c]) -{\rm tr}(J[b,[a,c]]) \big)\, \omega^n \\
&= 0 \,\, .
\end{align*}
Notice that, in virtue of \eqref{eq:tangent-space}, this implies that
\begin{equation} \label{eq:almkahlerdim(3)}
\mathcal{L}_{\hat{a}}\mathbb{\Omega} = 0
\end{equation}
and so, from \eqref{eq:almkahlerdim(1)} and \eqref{eq:almkahlerdim(3)} it follows that the basic vector fields are holomorphic Killing. Finally
\begin{align*}
{\rm d}\mathbb{\Omega}(\hat{a},\hat{b},\hat{c})_J &=
\big( \mathcal{L}_{\hat{a}}\big(\mathbb{\Omega}(\hat{b},\hat{c})\big)
+\mathcal{L}_{\hat{b}}\big(\mathbb{\Omega}(\hat{c},\hat{a})\big)
+\mathcal{L}_{\hat{c}}\big(\mathbb{\Omega}(\hat{a},\hat{b})\big)
-\mathbb{\Omega}([\hat{a},\hat{b}],\hat{c})
-\mathbb{\Omega}([\hat{b},\hat{c}],\hat{a})
-\mathbb{\Omega}([\hat{c},\hat{a}],\hat{b}) \big)_J \\
&= \mathbb{\Omega}([\hat{a},\hat{b}],\hat{c})_J
+\mathbb{\Omega}([\hat{b},\hat{c}],\hat{a})_J
+\mathbb{\Omega}([\hat{c},\hat{a}],\hat{b})_J \\
&= 2 \int_M \big({\rm tr}(J[[a,b],c]) +{\rm tr}(J[[b,c],a]) +{\rm tr}(J[[c,a],b])\big) \,\omega^n \\
&= 0
\end{align*}
and this concludes the proof.
\end{proof}

\begin{remark}
By \eqref{eq:tangent-space} and \eqref{eq:liebr}, it follows that the infinite dimensional Lie algebra $\mathcal{C}^{\infty}(M; \mathfrak{sp}(TM,\omega))^{\mathsf{K}}$ acts effectively and transitively on $\mathcal{J}_{\rm alm}(\omega)^{\mathsf{K}}$ (see {\it e.g.\ }\cite[Section 2]{Alek-Mich} for the notation). Moreover, by Proposition \ref{prop:almkahler}, the K\"ahler stucture defined in \eqref{eq:tautK} is invariant with respect to such Lie algebra action.
\end{remark}

Finally, we introduce the subset
\begin{equation} \label{def:JomegaT}
\mathcal{J}(\omega)^{\mathsf{K}} \coloneqq \big\{ J \in \mathcal{J}_{\rm alm}(\omega)^{\mathsf{K}} : \text{$J$ is integrable} \big\} \,\, .
\end{equation}
In the same spirit as \cite[Theorem 4.2]{fujiki}, one can show that $\mathcal{J}(\omega)^{\mathsf{K}}$ is an analytic subset of $\mathcal{J}_{\rm alm}(\omega)^{\mathsf{K}}$.
As already mentioned in Section \ref{sect:prel}, any element $J \in \mathcal{J}(\omega)^{\mathsf{K}}$ yields a locally conformally K\"ahler structure $(J,g_J = \omega(\_,J\_))$ on $M$, with fundamental $(1,1)$-form $\omega$, such that $\mathsf{K} \subset {\rm Aut}(M,J,g_J)$. 

\subsection{Linearization formulas} \hfill \par

The aim of this subsection is to compute the linearization of some geometric quantities related to the locally conformally K\"ahler structures induced by elements in $\mathcal{J}(\omega)^{\mathsf{K}}$. \smallskip

Fix $J \in \mathcal{J}(\omega)^{\mathsf{K}}$ and a direction $\hat{a}_J \in T_J\mathcal{J}_{\rm alm}(\omega)^{\mathsf{K}}$, for some $a \in \mathcal{C}^{\infty}(M;\mathfrak{sp}(TM,\omega))^{\mathsf{K}}$. For the sake of shortness, we set $g \coloneqq g_J$ and 
$$
\mathring{a} \coloneqq -J\hat{a}_J \,\, .
$$
An easy computation shows that the endomorphism $\mathring{a}$ is $g$-symmetric, $J$-anti-invariant and trace-free, {\it i.e.\ }
\begin{equation} \label{eq:propringa}
g(\mathring{a}(X),Y) = g(X,\mathring{a}(Y)) \,\, , \quad J\mathring{a} + \mathring{a}J=0 \,\, , \quad {\rm tr}(\mathring{a})=0 \,\, .
\end{equation}
In the following, for the sake of shortness, given any function $F$ defined on $\mathcal{J}_{\rm alm}(\omega)^{\mathsf{K}}$, we will denote by $F'$ the differential of $F$ at $J$ in the direction of $\hat{a}_J$. A straightforward computation proves the following

\begin{lemma} \label{lem:lin-g-theta}
The following equalities hold true:
\begin{align}
g'(X,Y) &= g(\mathring{a}(X),Y) \,\, , \label{eq:g'} \\
(\theta^{\sharp})' &= -\mathring{a}(\theta^{\sharp}) \,\, , \label{eq:theta'} \\
(|\theta|^2)' &= -g(\mathring{a}, \theta\otimes\theta^{\sharp}) \,\, . \label{eq:normth'}
\end{align}
\end{lemma}

As for the Riemannian scalar curvature, we have

\begin{proposition} 
The linearization of the Riemannian scalar curvature reads as
\begin{equation} \label{eq:lin-scal}
{\rm scal}' = {\rm d}^*(\delta \mathring{a}) +\tfrac{n-1}2 g(\mathring{a},\theta \otimes \theta^{\sharp})
\,\, .
\end{equation}
\end{proposition}

\begin{proof}
By \cite[Theorem 1.174]{besse}, \eqref{eq:propringa} and \eqref{eq:g'}, we know that
\begin{equation} \label{eq:lin-scal-dim(1)}
{\rm scal}' =  {\rm d}^*(\delta \mathring{a}) -g(\mathring{a},({\rm Ric})^{\sharp}) \,\, .
\end{equation}
Moreover, if we denote by ${\rm Ric}^\nabla$ the Weyl-Ricci tensor of $(M,J,g)$, then it follows that
$$
({\rm Ric}^\nabla)^{\sharp} -({\rm Ric})^{\sharp} = (n-1)(\delta^*\theta) +\tfrac{n-1}2 \theta \otimes \theta^{\sharp} -\tfrac12({\rm d}^*\theta){\rm Id} -\tfrac{n-1}2|\theta|^2{\rm Id}
$$
(see \eqref{eq:Ric-Ric} in Appendix \ref{appendixA}). Therefore, since ${\rm d}_{\theta}\theta=0$, by \eqref{eq:delta*} it follows that
\begin{equation} \label{eq:lin-scal-dim(2)}
-g(\mathring{a},({\rm Ric})^{\sharp}) = 
-g(\mathring{a},({\rm Ric}^\nabla)^{\sharp})
+(n-1)g(\mathring{a},\nabla(\theta^{\sharp}))
+\tfrac{n-1}2 g(\mathring{a},\theta \otimes \theta^{\sharp})
-\tfrac12({\rm d}^*\theta)g(\mathring{a},{\rm Id}) \,\, .
\end{equation}
We recall that the endomorphisms $({\rm Ric}^\nabla)^{\sharp}$ is $g$-symmetric and $J$-invariant (see Corollary \ref{cor:RicJinv} in Appendix \ref{appendixA}) and so, by \eqref{eq:propringa}, we get
\begin{equation} \label{eq:lin-scal-dim(3)}
g(\mathring{a},({\rm Ric}^\nabla)^{\sharp})=0 \,\, .
\end{equation}
In the same way, since $\mathcal{L}_VJ=0$ and $\theta^{\sharp} = JV$, by Lemma \ref{lem:holvf1} it follows that $\nabla(\theta^{\sharp})$ is $J$-invariant and so, by \eqref{eq:propringa}, we get
\begin{equation} \label{eq:lin-scal-dim(4)}
g(\mathring{a},\nabla(\theta^{\sharp}))=0 \,\, .
\end{equation}
Finally, again by \eqref{eq:propringa}, we get
\begin{equation} \label{eq:lin-scal-dim(5)}
g(\mathring{a},{\rm Id}) = {\rm tr}(\mathring{a}) = 0 \,\, .
\end{equation}
Therefore, \eqref{eq:lin-scal} follows from \eqref{eq:lin-scal-dim(1)}, \eqref{eq:lin-scal-dim(2)}, \eqref{eq:lin-scal-dim(3)}, \eqref{eq:lin-scal-dim(4)} and \eqref{eq:lin-scal-dim(5)}.
\end{proof}

Finally, regarding the codifferential of $\theta$, we get

\begin{lemma}
The linearization of ${\rm d}^*\theta$ reads as
\begin{equation} \label{eq:lin-deltaomega}
({\rm d}^*\theta)' = -(\delta\mathring{a})(\theta^{\sharp}) \,\, .
\end{equation}
\end{lemma}

\begin{proof}
Since ${\rm d}^*\theta = -{\rm tr}(D(\theta^{\sharp}))$, we get
\begin{equation} \label{eq:lin-deltaomega-dim(1)}
({\rm d}^*\theta)' = -{\rm tr}(D'(\theta^{\sharp})) -{\rm tr}(D(\theta^{\sharp})') \,\, .
\end{equation}
Fix a local $g$-orthonormal frame $(\tilde{e}_{\alpha})_{\alpha \in \{1,\ldots,2n\}}$ on $M$. Then, by \cite[Theorem 1.174]{besse} and \eqref{eq:g'}, it follows that
\begin{align*}
2{\rm tr}(D'(\theta^{\sharp}))
&= \sum_{\alpha=1}^{2n} 2g(D'(\tilde{e}_{\alpha},\theta^{\sharp}),\tilde{e}_{\alpha}) \\
&= \sum_{\alpha=1}^{2n} \big( g((D_{\tilde{e}_{\alpha}}\mathring{a})(\theta^{\sharp}),\tilde{e}_{\alpha}) +g((D_{\theta^{\sharp}}\mathring{a})(\tilde{e}_{\alpha}),\tilde{e}_{\alpha}) -g((D_{\tilde{e}_{\alpha}}\mathring{a})(\tilde{e}_{\alpha}),\theta^{\sharp}) \big) \\
&= {\rm tr}(D_{\theta^{\sharp}}\mathring{a}) \\
&= \mathcal{L}_{\theta^{\sharp}}\big({\rm tr}(\mathring{a})\big)
\end{align*}
and so, by \eqref{eq:propringa},
\begin{equation} \label{eq:lin-deltaomega-dim(2)}
{\rm tr}(D'(\theta^{\sharp})) = 0 \,\, .
\end{equation}
On the other hand, by \eqref{eq:theta'}, we have
\begin{equation} \label{eq:lin-deltaomega-dim(3)}
-{\rm tr}(D(\theta^{\sharp})') = {\rm tr}\big( D(\mathring{a}(\theta^{\sharp})) \big) = -(\delta\mathring{a})(\theta^{\sharp}) +g(D(\theta^{\sharp}),\mathring{a}) \,\, .
\end{equation}
Finally, a direct computation based on \eqref{def:Weyl} shows that
\begin{equation} \label{eq:lin-deltaomega-dim(4)}
D(\theta^{\sharp}) = \nabla(\theta^{\sharp}) +\tfrac12|\theta|^2{\rm Id}
\end{equation}
and so \eqref{eq:lin-deltaomega} follows by \eqref{eq:lin-scal-dim(4)}, \eqref{eq:lin-scal-dim(5)}, \eqref{eq:lin-deltaomega-dim(1)}, \eqref{eq:lin-deltaomega-dim(2)}, \eqref{eq:lin-deltaomega-dim(3)} and \eqref{eq:lin-deltaomega-dim(4)}.
\end{proof}

\medskip
\section{The moment map and the Futaki invariant} \label{sect:main}
\setcounter{equation} 0

Let $(M^{2n},\omega)$ be a compact, connected, smooth, locally conformally symplectic manifold of real dimension $2n$, with non-exact Lee form $\theta$, and assume that $\mathsf{K}$ is a compact, connected Lie group that acts effectively on $M$ and verifies \ref{(k1)}, \ref{(k2)}, \ref{(k3)} as in Section \ref{sect:actionK}. \smallskip

Let us consider the group of $\mathsf{K}$-invariant, special conformal automorphisms of $(M,\omega)$
\begin{equation} \label{def:specconfTGroup}
{\rm Aut}^{\star}(M,[\omega])^{\mathsf{K}} \coloneqq {\rm Aut}^{\star}(M,[\omega]) \cap {\rm Diff}(M)^{\mathsf{K}}
\end{equation}
and the subgroup of $\mathsf{K}$-invariant, twisted-Hamiltonian diffeomorphisms of $(M,\omega)$
\begin{equation} \label{def:twistedHamTGroup}
{\rm Ham}(M,[\omega])^{\mathsf{K}} \coloneqq {\rm Ham}(M,[\omega]) \cap {\rm Diff}(M)^{\mathsf{K}} \,\, ,
\end{equation}
with ${\rm Diff}(M)^{\mathsf{K}} \coloneqq \{\varphi \in {\rm Diff}(M) : \varphi \circ \psi = \psi \circ \varphi \text{ for any $\psi \in \mathsf{K}$} \}$. Analogously, we consider the Lie algebra of $\mathsf{K}$-invariant, special conformal vector fields of $(M,\omega)$
\begin{equation} \label{def:specconfTLie}
\mathfrak{aut}^{\star}(M,[\omega])^{\mathsf{K}} \coloneqq \mathfrak{aut}^{\star}(M,[\omega]) \cap \mathfrak{X}(M)^{\mathsf{K}}
\end{equation}
together with the Lie subalgebra of $\mathsf{K}$-invariant, twisted-Hamiltonian vector fields of $(M,\omega)$
\begin{equation} \label{def:twistedHamTLie}
\mathfrak{ham}(M,[\omega])^{\mathsf{K}} \coloneqq \mathfrak{ham}(M,[\omega]) \cap \mathfrak{X}(M)^{\mathsf{K}} \,\, .
\end{equation}
By the very definitions, it is immediate to check that both $\mathfrak{aut}^{\star}(M,[\omega])^{\mathsf{K}}$, $\mathfrak{ham}(M,[\omega])^{\mathsf{K}}$ verify the property \eqref{eq:g-isotopy} and that ${\rm Aut}^{\star}(M,[\omega])^{\mathsf{K}}$, ${\rm Ham}(M,[\omega])^{\mathsf{K}}$ are the generalized Lie transformation groups generated by them. The first ingredient to construct a moment map in our setting is the following

\begin{proposition} \label{prop:importante}
The following properties hold true.
\begin{itemize}
\item[a)] The Lie algebra $\mathfrak{aut}^{\star}(M,[\omega])^{\mathsf{K}}$ preserves $\omega$, {\it i.e.\ }
\begin{equation} \label{eq:autTpresomega}
\mathcal{L}_X\omega = 0 \quad \text{ for any }\,\, X \in \mathfrak{aut}^{\star}(M,[\omega])^{\mathsf{K}} \,\, .
\end{equation}
\item[b)] The map $\kappa$ defined in \eqref{eq:kappa} restricts to a linear isomorphism $\mathfrak{ham}(M,[\omega])^{\mathsf{K}} \simeq \mathcal{C}^{\infty}(M,\mathbb{R})^{\mathsf{K}}$.
\end{itemize}
\end{proposition}

\begin{proof}
Let $X \in \mathfrak{aut}^{\star}(M,[\omega])^{\mathsf{K}}$. Then, by \eqref{eq:commutator} we get
$$
{\rm d}_{\theta}(\theta(X)) = {\rm d}_{\theta}(\omega(V,X)) = -[V,X] \lrcorner \omega = 0 \,\, .
$$
Therefore, by Lemma \ref{lem:cohom0}, this means that $\theta(X) = 0$ and so \eqref{eq:autTpresomega} follows by \eqref{eq:Lomega}.

Fix now $X \in \mathfrak{ham}(M,[\omega])$. By definition, the action of $\mathsf{K}$ preserves $\omega$, and hence it preserves $\theta$ as well. Therefore, for any $\psi \in \mathsf{K}$, we get
$$
(\psi_* X) \lrcorner \omega = (\psi^{-1})^*(X\lrcorner \omega) = (\psi^{-1})^*{\rm d}_{\theta}(\kappa(X)) = {\rm d}_{\theta}(\kappa(X) \circ \psi^{-1})
$$
which implies, together with Lemma \ref{lem:cohom0}, that $X$ is $\mathsf{K}$-invariant if and only if $\kappa(X)$ is $\mathsf{K}$-invariant.
\end{proof}

In the following proposition, we introduce a natural action of the group ${\rm Aut}^{\star}(M,[\omega])^{\mathsf{K}}$ on the space $\mathcal{J}_{\rm alm}(\omega)^{\mathsf{K}}$ of $\mathsf{K}$-invariant, compatible, almost complex structures and we prove that it preserves the tau\-tolo\-gi\-cal symplectic structure $\mathbb{\Omega}$ defined in \eqref{eq:tautK} and the subset $\mathcal{J}(\omega)^{\mathsf{K}}$ defined in \eqref{def:JomegaT}.

\begin{proposition} \label{prop:twHamaction}
The generalized Lie transformation group ${\rm Aut}^{\star}(M,[\omega])^{\mathsf{K}}$ acts on $\mathcal{J}_{\rm alm}(\omega)^{\mathsf{K}}$ by pull-back
\begin{equation} \label{def:action}
{\rm Aut}^{\star}(M,[\omega])^{\mathsf{K}} \times \mathcal{J}_{\rm alm}(\omega)^{\mathsf{K}} \to \mathcal{J}_{\rm alm}(\omega)^{\mathsf{K}} \,\, , \quad (\varphi, J) \mapsto \varphi^*J \coloneqq ({\rm d}\varphi)^{-1} \circ J \circ {\rm d}\varphi
\end{equation}
and this action preserves both the subset $\mathcal{J}(\omega)^{\mathsf{K}} \subset \mathcal{J}_{\rm alm}(\omega)^{\mathsf{K}}$ and the symplectic form $\mathbb{\Omega}$. Moreover, for any $X \in \mathfrak{aut}^{\star}(M,[\omega])^{\mathsf{K}}$, the fundamental vector field $X^* \in \mathcal{C}^{\infty}\big(\mathcal{J}_{\rm alm}(\omega)^{\mathsf{K}}; T\mathcal{J}_{\rm alm}(\omega)^{\mathsf{K}}\big)$ associated to $X$ is
\begin{equation} \label{eq:fundvectors}
X^* = -\mathcal{L}_XJ \,\, .
\end{equation}
\end{proposition}

\begin{proof}
The action \eqref{def:action} is clearly well defined. Moreover, since the integrability of any $J \in \mathcal{J}_{\rm alm}(\omega)^{\mathsf{K}}$ is equivalent to the vanishing of the Nijenhuis tensor $N_J$, it follows that the subset $\mathcal{J}(\omega)^{\mathsf{K}}$ is ${\rm Aut}^{\star}(M,[\omega])^{\mathsf{K}}$-invariant. Formula \eqref{eq:fundvectors} follows directly from the definition of fundamental vector field and \eqref{def:action}. Finally, since ${\rm Aut}^{\star}(M,[\omega])^{\mathsf{K}}$ is connected by smooth arcs, in order to check that it preserves the symplectic form $\mathbb{\Omega}$, it is sufficient to prove that $\mathcal{L}_{X^*}\mathbb{\Omega} = 0$ for any $X \in \mathfrak{aut}^{\star}(M,[\omega])^{\mathsf{K}}$.

Therefore, fix $X \in \mathfrak{aut}^{\star}(M,[\omega])^{\mathsf{K}}$, $J \in \mathcal{J}_{\rm alm}(\omega)^{\mathsf{K}}$ and take two basic vector fields $\hat{a},\hat{b}$ on $\mathcal{J}_{\rm alm}(\omega)^{\mathsf{K}}$. A straightforward computation shows that
$$
[X^*,\hat{a}] = \widehat{\mathcal{L}_Xa}
$$
and then, by the infinite dimensional Cartan formula, we compute
\begin{align*}
(\mathcal{L}_{X^*}\mathbb{\Omega})(\hat{a},\hat{b})_J &= \big({\rm d}(X^* \lrcorner \mathbb{\Omega})\big)(\hat{a},\hat{b})_J \\
&= \mathcal{L}_{\hat{a}}\big(\mathbb{\Omega}(X^*,\hat{b})\big)
-\mathcal{L}_{\hat{b}}\big(\mathbb{\Omega}(X^*,\hat{a})\big)
-\mathbb{\Omega}(X^*,[\hat{a},\hat{b}]) \,\big|_J \\
&= -\mathbb{\Omega}(\widehat{\mathcal{L}_Xa},\hat{b})
-\mathbb{\Omega}(\hat{a},\widehat{\mathcal{L}_Xb})
+\mathbb{\Omega}(X^*,[\hat{a},\hat{b}]) \,\big|_J \\
&= -2\int_M \big( {\rm tr}(J[(\mathcal{L}_Xa),b]) +{\rm tr}(J[a,(\mathcal{L}_Xb)])
+{\rm tr}((\mathcal{L}_XJ)[a,b]) \big)\, \omega^n \\
&= -2\int_M \mathcal{L}_X \big( {\rm tr}(J[a,b]) \big) \, \omega^n \,\, .\\
\end{align*}
Therefore, by \eqref{eq:autTpresomega}, we get
$$
(\mathcal{L}_{X^*}\mathbb{\Omega})(\hat{a},\hat{b})_J = -2\int_M \mathcal{L}_X \big( {\rm tr}(J[a,b]) \big) \, \omega^n = -2\int_M \mathcal{L}_X \big( {\rm tr}(J[a,b]) \, \omega^n \big) =0
$$
and so the thesis follows from \eqref{eq:tangent-space}.
\end{proof}

Finally, in virtue of \eqref{eq:tangent-space}, Proposition \ref{prop:importante} and Proposition \ref{prop:twHamaction}, we are ready to prove the main result of this paper, that is the following

\begin{theorem} \label{thm:moment-map}
The map $\mu \colon \mathcal{J}(\omega)^{\mathsf{K}} \to \mathcal{C}^{\infty}(M;\mathbb{R})^{\mathsf{K}}$ defined by
\begin{equation} \label{eq:mu}
\mu \coloneqq {\rm scal}^{\rm Ch} +n\,{\rm d}^*\theta
\end{equation}
admits a smooth extension to an ${\rm Aut}^{\star}(M,[\omega])^{\mathsf{K}}$-equivariant map on the whole space $\mathcal{J}_{\rm alm}(\omega)^{\mathsf{K}}$ that verifies
$$
\int_M {\rm d}\mu|_J(\hat{a}_J)\,\kappa(X)\,\omega^n = -\tfrac12\,\mathbb{\Omega}_J(X^*_J,\hat{a}_J)
$$
for any $J \in \mathcal{J}(\omega)^{\mathsf{K}}$, $a \in \mathcal{C}^{\infty}(M;\mathfrak{sp}(TM,\omega))^{\mathsf{K}}$, $X \in \mathfrak{ham}(M,[\omega])^{\mathsf{K}}$.
\end{theorem}

\begin{proof}
We recall that, when $J$ is integrable, the Chern-scalar curvature can be expressed in terms of the Riemannian scalar curvature as
$$
{\rm scal}^{\rm Ch} = {\rm scal} -(n-1)({\rm d}^*\theta) +\tfrac{n-1}2|\theta|^2
$$
(see \eqref{eq:scalCh-scal} in Appendix \ref{appendixA}). Therefore, the function
$$
\mu \colon \mathcal{J}_{\rm alm}(\omega)^{\mathsf{K}} \to \mathcal{C}^{\infty}(M;\mathbb{R})^{\mathsf{K}} \,\, , \quad \mu \coloneqq {\rm scal} +{\rm d}^*\theta +\tfrac{n-1}2|\theta|^2
$$
extends \eqref{eq:mu} to the whole $\mathcal{J}_{\rm alm}(\omega)^{\mathsf{K}}$ and one can check that it is smooth by a direct computation in local coordinates. Moreover, since ${\rm Aut}^{\star}(M,[\omega])^{\mathsf{K}}$ preserves the $2$-form $\omega$, it follows that $g_{\varphi^*J} = \varphi^*g_J$ for any $J \in \mathcal{J}_{\rm alm}(\omega)^{\mathsf{K}}$, $\varphi \in {\rm Aut}^{\star}(M,[\omega])^{\mathsf{K}}$ and so the extended map $\mu$ is ${\rm Aut}^{\star}(M,[\omega])^{\mathsf{K}}$-equivariant.

Fix $J \in \mathcal{J}(\omega)^{\mathsf{K}}$, $a \in \mathcal{C}^{\infty}(M;\mathfrak{sp}(TM,\omega))^{\mathsf{K}}$, $X \in \mathfrak{ham}(M,[\omega])^{\mathsf{K}}$ and let $h \coloneqq \kappa(X) \in \mathcal{C}^{\infty}(M;\mathbb{R})^{\mathsf{K}}$. Set $g \coloneqq g_J$, $\mathring{a} \coloneqq -J\hat{a}_J$ and, for any function $F$ defined on $\mathcal{J}_{\rm alm}(\omega)^{\mathsf{K}}$, denote by $F'$ the differential of $F$ at $J$ in the direction of $\hat{a}_J$. Then, by \eqref{eq:normth'},\eqref{eq:lin-scal} and \eqref{eq:lin-deltaomega}, we get
\begin{align}
\langle h, \mu' \rangle_{L^2(M,g)}
&= \langle h, ({\rm scal})' \rangle_{L^2(M,g)} +\langle h, ({\rm d}^*\theta)' \rangle_{L^2(M,g)} +\tfrac{n-1}2\langle h, (|\theta|^2)' \rangle_{L^2(M,g)} \nonumber \\
&= \langle h, {\rm d}^*(\delta \mathring{a}) \rangle_{L^2(M,g)} +\tfrac{n-1}2\langle h, g(\mathring{a},\theta \otimes \theta^{\sharp}) \rangle_{L^2(M,g)} -\langle h, (\delta\mathring{a})(\theta^{\sharp}) \rangle_{L^2(M,g)} \nonumber \\
&\qquad -\tfrac{n-1}2\langle h, g(\mathring{a},\theta \otimes \theta^{\sharp}) \rangle_{L^2(M,g)} \nonumber \\
&= \langle \delta^*({\rm d}h), \mathring{a} \rangle_{L^2(M,g)} -\langle \delta^*(h\theta), \mathring{a} \rangle_{L^2(M,g)} \nonumber \\
&= \langle \delta^*({\rm d}_{\theta}h), \mathring{a} \rangle_{L^2(M,g)} \,\, . \label{eq:dimmomentum(1)}
\end{align}
Since ${\rm d}_\theta h = X\lrcorner\omega = g(JX,\_)$, it follows that $({\rm d}_\theta h)^{\sharp}=JX$ and so, from \eqref{eq:delta*} and ${\rm d}_\theta^2h=0$, we get
\begin{equation} \label{eq:dimmomentum(2)}
\delta^*({\rm d}_{\theta}h) = J \circ \nabla X +\tfrac12\theta(JX){\rm Id} \,\, .
\end{equation}
Moreover, since ${\rm d}_{\theta}(X \lrcorner \omega)=0$, by \eqref{eq:holvf1} and \eqref{eq:holvf2}
\begin{align*}
0 &= \omega(\nabla_YX,Z) +\omega(Y,\nabla_ZX) \\
&= g(J\nabla_YX,Z) +g(JY,\nabla_ZX) \\
&= g(\nabla_{JY}X,Z) +g(JY,\nabla_ZX) +g((\mathcal{L}_XJ)Y,Z)
\end{align*}
and so, using again \eqref{eq:holvf1}, we obtain
\begin{equation} \label{eq:dimmomentum(3)}
2\,{\rm Sym}_{g}(\nabla X) = -J \circ \mathcal{L}_XJ \,\, .
\end{equation}
Therefore, by \eqref{eq:tautK}, \eqref{eq:propringa}, \eqref{eq:dimmomentum(1)}, \eqref{eq:dimmomentum(2)}, \eqref{eq:dimmomentum(3)}, \eqref{eq:fundvectors} and the fact that $\hat{a}_J$ is $g$-symmetric, we obtain
$$
\langle h, \mu' \rangle_{L^2(M,g)}
= -\langle \nabla X, \hat{a}_J \rangle_{L^2(M,g)}
= -\tfrac12\langle \mathcal{L}_{X}J, J\hat{a}_J \rangle_{L^2(M,g)}
= -\tfrac12\,\mathbb{\Omega}_J(X^*_J,\hat{a}_J) \,\, ,
$$
which concludes the proof.
\end{proof}

As a direct consequence of Theorem \ref{thm:moment-map}, we get the following 

\begin{corollary} \label{cor:futaki}
The value
$$
\underline{\mu} \coloneqq \frac{\int_M \mu(J) \, \omega^n}{\int_M \omega^n}
$$
and the map
$$
\mathcal{F} \colon \mathfrak{z}(\mathfrak{k}) \to \mathbb{R} \,\, , \quad
\mathcal{F}(X) \coloneqq \int_M (\mu(J)-\underline{\mu}) \kappa(X) \, \omega^n
$$
are independent of $J$, in the connected components of ${\mathcal J}(\omega)^{\mathsf{K}}$.
\end{corollary}

\begin{proof}
We use the same notation as in the Proof of Theorem \ref{thm:moment-map}.
If $X \in \mathfrak{z}(\mathfrak{k})$ and $h\coloneqq\kappa(X)$, then by Lemma \ref{lem:holvf1}, \eqref{eq:dimmomentum(1)}, \eqref{eq:dimmomentum(2)}, it follows that
$$
\frac{{\rm d}}{{\rm d}t} \langle \mu(J_t), h \rangle_{L^2(M,g)}\big\vert_{t=0} = \langle \delta^*({\rm d}_{\theta}h), \mathring{a} \rangle_{L^2(M,g)} = -\langle \nabla X, \hat{a}_J \rangle_{L^2(M,g)} = 0 \,\, .
$$
In particular, notice that, since $\kappa(V)=-1$, we get the thesis.
\end{proof}

Notice that the existence of a complex structure $J \in {\mathcal J}(\omega)^{\mathsf{K}}$ such that $\mu(J)=\underline{\mu}$ forces $\mathcal{F}$ to vanish and thus it is an obstruction to the existence of these special locally conformally K\"ahler metrics.

\begin{remark}\label{rem:csck-covering}
Fix $J \in \mathcal J(\omega)$, and let $\widehat{g}_o=\widehat{\omega}_o(\_,\pi^*J\_)$ be the corresponding K\"ahler metric on the minimal symplectic covering $(\widehat{M},\widehat{\omega}_o)$, see Section \ref{sect:prel1}.
Then, by the standard formulas for conformal changes (see {\it e.g.\ }\cite[Theorem 1.159]{besse}) and \eqref{eq:mu} (see also \eqref{eq:scalCh-scal} in Appendix \ref{appendixA}), the Riemannian scalar curvature ${\rm scal}(\widehat{g}_o)$ can be computed in terms of $\mu(J)$ as
$$
{\rm scal}(\widehat{g}_o) = e^{w} \big( \mu(J) - 2n\,{\rm d}^*\theta - n(n-1) |\theta|^2 \big) \circ \pi \,\, ,
$$
where $\pi: \widehat{M} \to M$ denotes the projection and $f: \widehat{M} \to \mathbb{R}$ is the smooth function that verifies $\pi^*\omega = e^f\,\widehat{\omega}_o$. Therefore, the moment map $\mu$ defined in \eqref{eq:mu} does not correspond to the scalar curvature map on the minimal symplectic covering.
\end{remark}

\appendix

\section{} \label{appendixA} \setcounter{equation} 0

Let $(M,J,g)$ be a compact, connected, locally conformally K\"ahler manifold of dimension ${\rm dim_{\mathbb{R}}}M=2n$ with fundamental $2$-form $\omega \coloneqq g(J\_,\_)$ and Lee form $\theta$.
In this Appendix, we report some useful formulas that we use in the text. These are well known by the experts, but we decided to collect them here for the sake of readibility.

\subsection{The Chern connection of a locally conformally K\"ahler manifold} \label{subsec:appCh} \hfill \par

The {\it Chern connection} of $(M,J,g)$ is the linear connection defined by
\begin{equation} \label{def:Chern}
\nabla^{\rm Ch}_X Y = D_XY - \tfrac12 A^{\rm Ch}(X,Y) \,\, ,
\end{equation}
where $D$ denotes the {\it Levi-Civita connection} of $(M,g)$ and $A^{\rm Ch}(X,Y) \coloneqq {\rm d}\omega(JX,Y,\_)^{\sharp}$. By the locally conformally K\"ahler condition ${\rm d}\omega = \theta \wedge \omega$, it follows that
\begin{equation} \label{eq:ACh}
A^{\rm Ch}(X,Y) = \theta(JX)JY+\theta(Y)X-g(X,Y)\theta^{\sharp} \,\, .
\end{equation}
The Chern connection is characterized by the following properties (see {\it e.g.\ }\cite[p. 273]{gauduchon-bumi})
\begin{equation} \label{eq:propertiesCh}
\nabla^{\rm Ch} g = 0 \,\, , \quad \nabla^{\rm Ch} J =0 \,\, , \quad J(T^{\rm Ch}(X,Y))=T^{\rm Ch}(JX,Y)=T^{\rm Ch}(X,JY) \,\, , 
\end{equation}
where $T^{\rm Ch}(X,Y) \coloneqq \nabla^{\rm Ch}_XY -\nabla^{\rm Ch}_YX - [X,Y]$ is the {\it torsion tensor of $\nabla^{\rm Ch}$}. By \eqref{def:Chern} and \eqref{eq:ACh}, the tensor $T^{\rm Ch}$ takes the form (compare with \cite[p. 500]{gauduchon-mathann})
\begin{equation} \label{eq:TCh}
T^{\rm Ch}(X,Y) = \tfrac12\big(\theta(X)Y -\theta(Y)X -\theta(JX)JY +\theta(JY)JX\big) \,\, .
\end{equation}
Moreover, by \eqref{def:Chern}, \eqref{eq:ACh} and \eqref{eq:propertiesCh}, the covariant derivative $DJ$ can be expressed in terms of $\theta$ as
\begin{equation} \label{eq:DJ}
(D_YJ)(X) = \tfrac12\big(\theta(JX)Y -\theta(X)JY -g(JX,Y)\theta^{\sharp} +g(X,Y)J\theta^{\sharp}\big) \,\, .
\end{equation}
For later use, we also mention that a straightforward computation based on \eqref{def:Chern}, \eqref{eq:ACh} and \eqref{eq:TCh} shows that the covariant derivative $\nabla^{\rm Ch}T^{\rm Ch}$ can be expressed in terms of $\theta$ as
\begin{equation} \label{eq:nablaChTCh} \begin{aligned}
2g((\nabla^{\rm Ch}_ZT^{\rm Ch})&(X,Y),W) = \\
=& -(D_Z\theta)(Y)g(X,W) -\tfrac12\big(\theta(Y)\theta(Z) +\theta(JY)\theta(JZ) -|\theta|^2g(Y,Z)\big)g(X,W) \\
& +(D_Z\theta)(X)g(Y,W) +\tfrac12\big(\theta(X)\theta(Z) +\theta(JX)\theta(JZ) -|\theta|^2g(X,Z)\big)g(Y,W) \\
& +(D_Z\theta)(JY)g(JX,W) +\tfrac12\big(\theta(JY)\theta(Z) -\theta(Y)\theta(JZ) -|\theta|^2g(JY,Z)\big)g(JX,W) \\
& -(D_Z\theta)(JX)g(JY,W) -\tfrac12\big(\theta(JX)\theta(Z) -\theta(X)\theta(JZ) -|\theta|^2g(JX,Z)\big)g(JY,W) \,\, .
\end{aligned} \end{equation}
We denote by
\begin{equation} \label{def:Chcurv}
\Omega^{\rm Ch}(X,Y) \coloneqq \nabla^{\rm Ch}_{[X,Y]}-[\nabla^{\rm Ch}_X,\nabla^{\rm Ch}_Y]
\end{equation}
the {\it Chern curvature tensor} of $(M,J,g)$, by
\begin{equation} \label{def:fakescalCh}
\widetilde{\rm scal}{}^{\rm Ch} \coloneqq \sum_{\alpha,\beta=1}^{2n} g(\Omega^{\rm Ch}(\tilde{e}_{\alpha},\tilde{e}_{\beta})\tilde{e}_{\alpha},\tilde{e}_{\beta})
\end{equation}
the real trace of $\Omega^{\rm Ch}$ and by 
\begin{equation} \label{def:scalCh}
{\rm scal}^{\rm Ch} \coloneqq 2\sum_{i,j=1}^{n} g(\Omega^{\rm Ch}(e_i,Je_i)e_j,Je_j)
\end{equation}
the {\it Chern-scalar curvature} of $(M,J,g)$. Here, $(\tilde{e}_{\alpha})_{\alpha \in \{1,\ldots,2n\}} = (e_i,Je_i)_{i \in \{1,\ldots,n\}}$ is any local $(J,g)$-unitary basis on $M$. \smallskip

On a locally conformally K\"ahler manifold, the function $\widetilde{\rm scal}{}^{\rm Ch}$ and the Riemannian scalar curvature are related as follows (compare with \cite[Eq. (32)]{gauduchon-mathann}).

\begin{lemma}
The function $\widetilde{\rm scal}{}^{\rm Ch}$ defined in \eqref{def:fakescalCh} can be expressed in terms of the Riemannian scalar curvature ${\rm scal}$ and $\theta$ as
\begin{equation} \label{eq:fakescalCh-scal}
\widetilde{\rm scal}{}^{\rm Ch} = {\rm scal} -2(n-1)({\rm d}^*\theta) -(n-\tfrac32)(n-1)|\theta|^2 \,\, .
\end{equation}
\end{lemma}

\begin{proof}
For the sake of notation, we set
\begin{equation} \label{def:fakeRicciCh}
\widetilde{\rm Ric}{}^{\rm Ch}(X,Y) \coloneqq {\rm tr} ( Z \mapsto \Omega^{\rm Ch}(X,Z)Y ) \,\, .
\end{equation}
Then, by \eqref{def:Chern} and \eqref{def:Chcurv}
\begin{align*}
\Omega^{\rm Ch}(X,Z) &= \nabla^{\rm Ch}_{[X,Z]} -[\nabla^{\rm Ch}_X, \nabla^{\rm Ch}_Z] \\
&= D_{[X,Z]} -[D_X, D_Z] +\tfrac12[D_X, A^{\rm Ch}(Z,\_)] -\tfrac12[D_Z, A^{\rm Ch}(X,\_)] -\tfrac12A^{\rm Ch}([X,Z],\_) \\
&\qquad -\tfrac14[A^{\rm Ch}(X,\_), A^{\rm Ch}(Z,\_)] \\
&= {\rm Rm}(X,Z) +\tfrac12(D_XA^{\rm Ch})(Z,\_) -\tfrac12(D_ZA^{\rm Ch})(X,\_) -\tfrac14[A^{\rm Ch}(X,\_), A^{\rm Ch}(Z,\_)]
\end{align*}
and so, by \eqref{def:fakeRicciCh}, we get
\begin{align*}
\widetilde{\rm Ric}{}^{\rm Ch}(X,Y) -{\rm Ric}(X,Y) &= \tfrac12\sum_{\alpha=1}^{2n}g((D_XA^{\rm Ch})(\tilde{e}_{\alpha},Y),\tilde{e}_{\alpha}) -\tfrac12\sum_{\alpha=1}^{2n}g((D_{\tilde{e}_{\alpha}}A^{\rm Ch})(X,Y),\tilde{e}_{\alpha}) \\
&\qquad -\tfrac14\sum_{\alpha=1}^{2n}g(A^{\rm Ch}(X,A^{\rm Ch}(\tilde{e}_{\alpha},Y)),\tilde{e}_{\alpha}) +\tfrac14\sum_{\alpha=1}^{2n}g(A^{\rm Ch}(\tilde{e}_{\alpha},A^{\rm Ch}(X,Y)),\tilde{e}_{\alpha}) \,\, ,
\end{align*}
where $(\tilde{e}_{\alpha})_{\alpha \in \{1,\ldots,2n\}}$ is any local $g$-orthonormal basis on $M$. Since
\begin{equation} \label{eq:d*omega}
{\rm d}^*\omega = -(n-1)J\theta \,\, ,
\end{equation}
by \eqref{eq:ACh}, \eqref{eq:DJ} and \eqref{eq:d*omega} it follows that
\begin{align*}
\sum_{\alpha=1}^{2n}g((D_XA^{\rm Ch})(\tilde{e}_{\alpha},Y),\tilde{e}_{\alpha}) &= 2(n-1)(D_X\theta)(Y) \,\, , \\
\sum_{\alpha=1}^{2n}g((D_{\tilde{e}_{\alpha}}A^{\rm Ch})(X,Y),\tilde{e}_{\alpha}) &= (D_X\theta)(Y) +(D_{JY}\theta)(JX) +\tfrac12\theta(X)\theta(Y) +(n-\tfrac12)\theta(JX)\theta(JY) \,\, , \\
&\quad +({\rm d}^*\theta) g(X,Y) -\tfrac12|\theta|^2g(X,Y) \,\, ,\\
\sum_{\alpha=1}^{2n}g(A^{\rm Ch}(X,A^{\rm Ch}(\tilde{e}_{\alpha},Y)),\tilde{e}_{\alpha}) &= \theta(X)\theta(Y) +\theta(JX)\theta(JY) -|\theta|^2g(X,Y) \,\, , \\
\sum_{\alpha=1}^{2n}g(A^{\rm Ch}(\tilde{e}_{\alpha},A^{\rm Ch}(X,Y)),\tilde{e}_{\alpha}) &= 2(n-1)\theta(X)\theta(Y) +2(n-1)\theta(JX)\theta(JY) -2(n-1)|\theta|^2g(X,Y)
\end{align*}
and so
\begin{multline} \label{eq:fakeRicCh-Ric}
\widetilde{\rm Ric}{}^{\rm Ch}(X,Y) -{\rm Ric}(X,Y) = (n-\tfrac32)(D_X\theta)(Y) -\tfrac12(D_{JY}\theta)(JX) -\tfrac12({\rm d}^*\theta) g(X,Y) \\
+(\tfrac{n}2-1)\theta(X)\theta(Y) -\tfrac12\theta(JX)\theta(JY) -(\tfrac{n}2-1)|\theta|^2g(X,Y) \,\, .
\end{multline}
By taking the trace of \eqref{eq:fakeRicCh-Ric}, we get \eqref{eq:fakescalCh-scal}.
\end{proof}

Analogously, the function $\widetilde{\rm scal}{}^{\rm Ch}$ and the Chern-scalar curvature are related, in the locally conformally K\"ahler setting, as follows (compare with \cite[Eq. (19)]{gauduchon-mathann}).

\begin{lemma}
The function $\widetilde{\rm scal}{}^{\rm Ch}$ defined in \eqref{def:fakescalCh} can be expressed in terms of the Chern scalar curvature ${\rm scal}^{\rm Ch}$ and $\theta$ as
\begin{equation} \label{eq:fakescalCh-scalCh}
\widetilde{\rm scal}{}^{\rm Ch} = {\rm scal}^{\rm Ch} -(n-1)({\rm d}^*\theta) -(n-1)^2|\theta|^2 \,\, .
\end{equation}
\end{lemma}

\begin{proof}
Fix a local $(J,g)$-unitary basis $(\tilde{e}_{\alpha})_{\alpha \in \{1,\ldots,2n\}} = (e_i,Je_i)_{i \in \{1,\ldots,n\}}$ on $M$. Then, by \eqref{eq:propertiesCh}, \eqref{eq:nablaChTCh}, \eqref{def:fakescalCh}, \eqref{def:scalCh} the first Bianchi Identity \cite[Theorem III.5.3]{KN1} and the properties of the Chern curvature tensor, one can compute
\begin{align*}
&\widetilde{\rm scal}{}^{\rm Ch} - {\rm scal}^{\rm Ch} = \\
&= \sum_{\alpha,\beta=1}^{2n} g(\Omega^{\rm Ch}(\tilde{e}_{\alpha},\tilde{e}_{\beta})\tilde{e}_{\alpha},\tilde{e}_{\beta}) -2 \sum_{i,j=1}^{n} g(\Omega^{\rm Ch}(e_i,Je_i)e_j,Je_j) \\
&= \sum_{i,j=1}^{n} \Big( g(\Omega^{\rm Ch}(e_i,e_j)e_i,e_j)
+g(\Omega^{\rm Ch}(e_i,Je_j)e_i,Je_j)
+g(\Omega^{\rm Ch}(Je_i,e_j)Je_i,e_j) \\
&\qquad\qquad +g(\Omega^{\rm Ch}(Je_i,Je_j)Je_i,Je_j)
-2g(\Omega^{\rm Ch}(e_i,Je_i)e_j,Je_j) \Big) \\
&= \sum_{i,j=1}^{n} \Big(-2g(\Omega^{\rm Ch}(Je_i,e_j)e_i,Je_j)
-2g(\Omega^{\rm Ch}(e_i,Je_i)e_j,Je_j)
-2g(\Omega^{\rm Ch}(e_j,e_i)Je_i,Je_j) \Big) \\
&= \sum_{i,j=1}^{n} \Big( 2g(T^{\rm Ch}(T^{\rm Ch}(Je_i,e_j),e_i),Je_j)
+2g(T^{\rm Ch}(T^{\rm Ch}(e_i,Je_i),e_j),Je_j)
+2g(T^{\rm Ch}(T^{\rm Ch}(e_j,e_i),Je_i),Je_j) \\
&\qquad\qquad +2g((\nabla^{\rm Ch}_{Je_i}T^{\rm Ch})(e_j,e_i),Je_j)
+2g((\nabla^{\rm Ch}_{e_i}T^{\rm Ch})(Je_i,e_j),Je_j)
+2g((\nabla^{\rm Ch}_{e_j}T^{\rm Ch})(e_i,Je_i),Je_j) \Big) \\
&= \sum_{i,j=1}^{n} \Big( 2g((\nabla^{\rm Ch}_{e_i}T^{\rm Ch})(e_i,e_j),e_j) -2g((\nabla^{\rm Ch}_{Je_i}T^{\rm Ch})(e_i,e_j),Je_j) \Big) \\
&= \sum_{i,j=1}^{n} \Big( \big((D_{e_i}\theta)(e_i) +(D_{Je_i}\theta)(Je_i)\big)g(e_j,e_j) +\big(\theta(e_i)\theta(e_i) +\theta(Je_i)\theta(Je_i) -|\theta|^2g(e_i,e_i)\big)g(e_j,e_j) \\
&\qquad\qquad -\big((D_{e_i}\theta)(e_j)+(D_{Je_i}\theta)(Je_j)\big)g(e_i,e_j) -\big(\theta(e_i)\theta(e_j) +\theta(Je_i)\theta(Je_j) -|\theta|^2g(e_i,e_j)\big)g(e_i,e_j) \Big) \\
&= -n({\rm d}^*\theta) -n(n-1)|\theta|^2 +({\rm d}^*\theta) +(n-1)|\theta|^2 \\
&= -(n-1)({\rm d}^*\theta) -(n-1)^2|\theta|^2
\end{align*}
which yields the proof.
\end{proof}

As a direct consequence of \eqref{eq:fakescalCh-scal} and \eqref{eq:fakescalCh-scalCh}, we obtain the following relation between the Riemannian scalar curvature and the Chern scalar curvature in the locally conformally K\"ahler setting (compare with \cite[Eq. (33)]{gauduchon-mathann})

\begin{corollary}
The difference between the Riemannian scalar curvature and the Chern scalar curvature can be expressed in terms of $\theta$ as
\begin{equation} \label{eq:scalCh-scal}
{\rm scal} -{\rm scal}^{\rm Ch} = (n-1)({\rm d}^*\theta) -\tfrac{n-1}2|\theta|^2 \,\, .
\end{equation}
\end{corollary}

\subsection{The Weyl connection of a locally conformally K\"ahler manifold} \hfill \par

Let $\nabla$ be the Weyl connection of $(M,J,g)$ as in \eqref{def:Weyl}. We denote now by
\begin{equation} \label{def:Weylcurv}
\Omega^\nabla(X,Y) \coloneqq \nabla_{[X,Y]}-[\nabla_X,\nabla_Y]
\end{equation}
the {\it Weyl curvature tensor} of $(M,J,g)$. Here, we collect some well-known properties of the Weyl curvature and, for the sake of completeness, we give some details on the proof.

\begin{lemma} \label{lem:weyl-curvature}
The Weyl-curvature tensor $\Omega^\nabla$ satisfies the following properties:
\begin{gather}
\Omega^\nabla(X,Y)Z +\Omega^\nabla(Y,Z)X +\Omega^\nabla(Z,X)Y = 0 \,\, , \label{eq:omegabianchi1} \\
g(\Omega^\nabla(X,Y)Z,W) +g(Z,\Omega^\nabla(X,Y)W) = 0 \,\, , \label{eq:omegaskew} \\
\Omega^\nabla(X,Y) \circ J = J \circ \Omega^\nabla(X,Y) \,\, , \label{eq:omegaJinv} \\
\Omega^\nabla(JX,JY) = \Omega^\nabla(X,Y) \,\, . \label{eq:omega11}
\end{gather}
\end{lemma}

\begin{proof}
The Bianchi identity \eqref{eq:omegabianchi1} holds since the Weyl connection is torsion-free \cite[Theorem III.5.3]{KN1}. Moreover, a straightforward computation shows that
\begin{align*}
&g(\Omega^{\nabla}(X,Y)Z,W) +g(Z,\Omega^{\nabla}(X,Y)W) \\
&\quad = g(\nabla_{[X,Y]}Z,W) +g(Z,\nabla_{[X,Y]}W) -g(\nabla_X\nabla_YZ,W) -g(Z,\nabla_X\nabla_YW) \\
&\quad\quad +g(\nabla_Y\nabla_XZ,W) +g(Z,\nabla_Y\nabla_XW) \\
&\quad = \mathcal{L}_{[X,Y]}(g(Z,W)) -\theta([X,Y])g(Z,W) -\mathcal{L}_X\big(g(\nabla_YZ,W)+g(Z,\nabla_YW)\big) \\
&\quad\quad +\theta(X)\big(g(\nabla_YZ,W) +g(Z,\nabla_YW)\big) +\mathcal{L}_Y\big(g(\nabla_XZ,W)+g(Z,\nabla_XW)\big) \\
&\quad\quad -\theta(Y)\big(g(\nabla_XZ,W) +g(Z,\nabla_XW)\big) \\
&\quad = \big(\mathcal{L}_X(\theta(Y)) -\mathcal{L}_Y(\theta(X)) -\theta([X,Y])\big)g(Z,W) -\theta(X)\big(\mathcal{L}_X(g(Z,W)) -g(\nabla_YZ,W) -g(Z,\nabla_YW)\big) \\
&\quad\quad +\theta(Y)\big(\mathcal{L}_X(g(Z,W)) -g(\nabla_XZ,W) -g(Z,\nabla_XW)\big) \\
&\quad = {\rm d}\theta(X,Y)g(Z,W)
\end{align*}
and so, as $\theta$ is closed, this proves \eqref{eq:omegaskew}. Formula \eqref{eq:omegaJinv} follows directly from $\nabla J=0$. As for \eqref{eq:omega11} we notice that, when $\mathbb{C}$-linearly extended to the complexified tangent bundle $T^\mathbb{C}M$ and then restricted to the holomorphic tangent bundle, the $(0,1)$-component of the Weyl connection coincides with the $(0,1)$-component of the Chern connection. Indeed
$$
A^{\rm Ch}(\overline{\eta},\xi) = \theta(J\overline{\eta})J\xi +\theta(\xi)\overline{\eta} -g(\overline{\eta},\xi)\theta^{\#} = \theta(\overline{\eta})\xi +\theta(\xi)\overline{\eta} -g(\overline{\eta},\xi)\theta^{\#} = A(\overline{\eta},\xi)
$$
for any pair of complex $(1,0)$-vector fields $\xi,\eta$ on $(M,J)$. This implies that the $(0,1)$-component of the Weyl connection is the Cauchy-Riemann operator of the holomorphic tangent bundle of $(M,J)$, and therefore the $(0,2)$-component of $\Omega^{\nabla}$ vanishes. Moreover, since $\Omega^{\nabla}$ is skew-Hermitian, its $(2,0)$-component vanishes, and this completes the proof.
\end{proof}

By taking the trace of the Weyl curvature tensor \eqref{def:Weylcurv}, we define the {\it Weyl-Ricci tensor}
\begin{equation} \label{def:RicciWeyl}
{\rm Ric}^\nabla(X,Y) \coloneqq {\rm tr} ( Z \mapsto \Omega^\nabla(X,Z)Y )
\end{equation}
and we notice that Lemma \ref{lem:weyl-curvature} has the following immediate

\begin{corollary} \label{cor:RicJinv}
The Weyl-Ricci tensor ${\rm Ric}^\nabla$ is symmetric and $J$-invariant, {\it i.e.\ }
$$
{\rm Ric}^\nabla(X,Y) = {\rm Ric}^\nabla(Y,X) \,\, , \quad {\rm Ric}^\nabla(JX,JY) = {\rm Ric}^\nabla(X,Y) \,\, . 
$$
\end{corollary}

Moreover, the difference between the two Ricci tensors ${\rm Ric}^\nabla$, ${\rm Ric}$ can be easily expressed in terms of the Lee form $\theta$. More precisely

\begin{proposition}
The following equality holds true:
\begin{equation} \label{eq:Ric-Ric}
({\rm Ric}^\nabla)^{\sharp} -({\rm Ric})^{\sharp} = (n-1)(\delta^*\theta) +\tfrac{n-1}2 \theta \otimes \theta^{\sharp} -\tfrac12({\rm d}^*\theta){\rm Id} -\tfrac{n-1}2|\theta|^2{\rm Id} \,\, .
\end{equation}
\end{proposition}

\begin{proof}
Fix a local $g$-orthonormal frame $(\tilde{e}_{\alpha})_{\alpha \in \{1,\ldots,2n\}}$ on $M$. A straightforward computation based on \eqref{def:Weyl}, \eqref{def:Weylcurv} and \eqref{def:Riemcurv} shows that
$$
\Omega^{\nabla}(X,Z) -{\rm Rm}(X,Z) = \tfrac12(D_XA)(Z,\_) -\tfrac12(D_ZA)(X,\_) -\tfrac14[A(X,\_), A(Z,\_)]
$$
and so, by \eqref{def:Riemcurv} and \eqref{def:RicciWeyl}
\begin{equation} \label{eq:Ric-Ric(1)} \begin{aligned}
{\rm Ric}^\nabla(X,Y) -{\rm Ric}(X,Y) &= \tfrac12\sum_{\alpha=1}^{2n} g((D_XA)(\tilde{e}_{\alpha},Y),\tilde{e}_{\alpha}) -\tfrac12\sum_{\alpha=1}^{2n} g((D_{\tilde{e}_{\alpha}}A)(X,Y),\tilde{e}_{\alpha}) \\
&\qquad -\tfrac14\sum_{\alpha=1}^{2n} g(A(X,A(\tilde{e}_{\alpha},Y)),\tilde{e}_{\alpha}) +\tfrac14\sum_{\alpha=1}^{2n} g(A(\tilde{e}_{\alpha},A(X,Y)),\tilde{e}_{\alpha}) \,\, .
\end{aligned} \end{equation}
By \eqref{def:Weyl}, it follows that
$$
(D_XA)(Y,Z) = (D_X\theta)(Y)Z +(D_X\theta)(Z)Y -g(Y,Z)D_X\theta^{\sharp} \,\, ,
$$
and so a straightforward computation shows that
\begin{equation} \label{eq:Ric-Ric(2)} \begin{aligned}
\sum_{\alpha=1}^{2n} g((D_XA)(\tilde{e}_{\alpha},Y),\tilde{e}_{\alpha}) &= 2n(D_X\theta)(Y) \,\, ,\\
\sum_{\alpha=1}^{2n} g((D_{\tilde{e}_{\alpha}}A)(X,Y),\tilde{e}_{\alpha}) &= (D_X\theta)(Y) +(D_Y\theta)(X) +({\rm d}^*\theta)g(X,Y) \,\, , \\
\sum_{\alpha=1}^{2n} g(A(X,A(\tilde{e}_{\alpha},Y)),\tilde{e}_{\alpha}) &= 2(n+1)\theta(X)\theta(Y) -2|\theta|^2g(X,Y) \,\, ,\\
\sum_{\alpha=1}^{2n} g(A(\tilde{e}_{\alpha},A(X,Y)),\tilde{e}_{\alpha}) &= 4n\theta(X)\theta(Y) -2n|\theta|^2g(X,Y) \,\, .
\end{aligned} \end{equation}
Finally, since $\theta$ is closed, it follows that $D\theta$ is symmetric and so
\begin{equation} \label{eq:Ric-Ric(4)}
(D\theta)^{\sharp} = D(\theta^{\sharp}) = \delta^*\theta \,\, .
\end{equation}
Therefore, \eqref{eq:Ric-Ric} follows from \eqref{eq:Ric-Ric(1)}, \eqref{eq:Ric-Ric(2)} and \eqref{eq:Ric-Ric(4)}.
\end{proof}

\bigskip

\noindent {\bf Conflict of Interest} The authors have no conflict of interest to declare that are relevant to this article.

\end{document}